\documentclass[a4paper,10pt]{article}
\usepackage[utf8]{inputenc}
\usepackage{import}

\usepackage{fullpage}

\usepackage{layouts}

\usepackage[english]{babel}

\usepackage{amsmath,amssymb,amsthm}
\usepackage{mathtools}
\usepackage{mathrsfs}
\usepackage{xfrac}

\usepackage{pdflscape}
\usepackage{tabularx}
\usepackage{booktabs}         % Tabelle
\usepackage{multicol}
\usepackage{multirow}
\usepackage{xparse}
\usepackage{xspace}
\usepackage{makecell}

\usepackage[usenames,dvipsnames,table]{xcolor}

\usepackage{footnote} % load it after xcolor becaus its in conflict with it
\usepackage{graphicx}
\usepackage{placeins}
\usepackage{svg}
\svgpath{{./images/}}

\definecolor{ultramarine}{RGB}{0,32,96}
\usepackage[colorlinks=true]{hyperref}         % bibliografia cliccabile
\usepackage{cleveref}
\usepackage{caption}          % per le didascalie delle tabelle

\usepackage[color=gray, backgroundcolor=gray!10, textwidth=2cm, textsize=footnotesize,disable]{todonotes}
\usepackage{setspace}
\usepackage{subcaption}

\usepackage{framed} % Framing content

\usepackage{enumitem}

\definecolor{DarkerGreen}{RGB}{0,170,0}
\definecolor{DarkerRed}{RGB}{170,0,0}

\definecolor{myRed}{rgb}{0.450385, 0.157961, 0.217975}
\definecolor{myBlue}{rgb}{0.139681, 0.311666, 0.550652}

\definecolor{myKcolor}{rgb}{0, 0.411765, 0.572549}
\definecolor{myEcolor}{rgb}{0, 0, 0}

\newcommand{\Fig}[1]{Fig.~\ref{#1}}

\NewDocumentCommand\sstr{o}{\IfNoValueTF{#1}{\sigma}{\sigma\ar{#1}}}

\newcommand{\scristr}{\bar{\sstr}}
\newcommand{\Gc}{G_{\l{c}}} % fracture toughness
\newcommand{\Gf}{G_{\l{f}}} % fracture toughness

\newcommand\res{\mathop{\hbox{\vrule height 7pt width .5pt depth 0pt
			\vrule height .5pt width 6pt depth 0pt}}\nolimits}

\newcommand\eps{\varepsilon}

\newcommand\R{\mathbb{R}}

\newcommand\calA{\mathcal{A}}
\newcommand\calL{\mathcal{L}}
\newcommand\calH{\mathcal{H}}

\newcommand\calulu{\mathfrak{U}}
\newtheorem{theorem}{Theorem}[section]

\newtheorem{proposition}[theorem]{Proposition}
\newtheorem{lemma}[theorem]{Lemma}
\newtheorem{remark}[theorem]{Remark}
\newtheorem{corollary}[theorem]{Corollary}

\numberwithin{equation}{section}
\usepackage{fancyhdr}
\newcommand\Functeps{\mathcal F_\eps}

\newcommand\dx{{\mathrm d}x}

\newcommand\dt{{\mathrm d}t}
\newcommand\dd{{\mathrm d}}

\newcommand\PotDann{\omega}
\newcommand\FailureS{\overline{\varsigma}}
\newcommand{\hsigmabarra}{h_{\FailureS}}

\newcommand\sfratt{s_{\mathrm{frac}}}
\newcommand\Functu{\mathscr{F}}

\newcommand\hatf{l}
\newcommand\hatk{\vartheta}

\graphicspath{
  {images/}
  {images/mathematica/}
  {/home/roberto/Insync/OneDrive/UniPi/ZoomNotes-backup/Research/Fracture/}}

\begin{document}

% _________________________________________________________________ Front Matter
%\begin{frontmatter}

\begin{center}
  {\Large
{Phase-field modelling of cohesive fracture. Part II: Reconstruction of the cohesive law}}\\[5mm]
{\today}\\[5mm]
Roberto Alessi$^{1}$, Francesco Colasanto$^{2}$ and Matteo Focardi$^{2}$ \\[2mm]
{\em $^{1}$ DICI, Universit\`a di Pisa\\ 56122, Pisa, Italy}\\[1mm]
{\em $^{2}$ DiMaI, Universit\`a di Firenze\\ 50134 Firenze, Italy}
\\[3mm]
    \begin{minipage}[c]{0.8\textwidth}

%\title{Phase-field modelling of cohesive fracture. Part II: Reconstruction of the cohesive law}

%\begin{abstract}
% The main aim of this three-part work is to provide a unified consistent framework for the phase-field modelling of cohesive fracture.
% Building on the theoretical foundations of the first paper, where {$\Gamma$-convergence} results have been derived, this second paper presents a systematic procedure for constructing phase-field models that reproduce prescribed cohesive laws. By either selecting the degradation function and determining the damage potential or vice versa, we enable the derivation of multiple phase-field models that exhibit the same cohesive fracture behavior but differ in their localized phase-field evolution. This methodology provides a flexible and rigorous strategy for tailoring phase-field models to specific cohesive responses, as shown by the several examples worked out.
% The mechanical responses associated with these examples, highlighting their features and validating the theoretical results, are investigated in the third paper from a more engineering-oriented and applied perspective.
%\end{abstract}

This is the second paper of a three-part work the main aim of which is to provide a unified consistent framework for the
phase-field modelling of cohesive fracture.
Building on the theoretical foundations of the first paper, where {$\Gamma$-convergence} results have been derived, this second paper presents a systematic procedure for constructing phase-field models that reproduce prescribed cohesive laws. By either selecting the degradation function and determining the damage potential or vice versa, we enable the derivation of multiple phase-field models that exhibit the same cohesive fracture behavior but differ in their localized phase-field evolution. This methodology provides a flexible and rigorous strategy for tailoring phase-field models to specific cohesive responses, as shown by the several examples worked out.
The mechanical responses associated with these examples, highlighting their features and validating the theoretical results, are investigated in the third paper from a more engineering-oriented and applied perspective.

\end{minipage}
\end{center}

%\end{frontmatter}

\tableofcontents

%\clearpage

% _______________________________________________________________
%\section*{To discuss and ToDo}
%\begin{itemize}
% \item Produco delle figurelle per gli esempi?
% \item Generalizziamo gli esempi?
%\end{itemize}

% _______________________________________________________________
\section{Introduction}
\label{sec_cpflit}
\FloatBarrier

\subsection{Background and Motivation}
% ------------------------------------------------------------------------------

Accurately predicting fracture initiation and propagation is crucial for designing resilient mechanical systems and preventing structural failures. Griffith’s brittle fracture theory \cite{Griffith1921}, despite its simplicity and widespread use, assumes that fracture energy dissipates entirely upon crack formation, neglecting any interaction between crack surfaces (\Fig{fig_BCL}-a). This leads to two major limitations: the inability to predict crack initiation in a pristine material \cite{marigo2010,Tanne2018,Kumar2020} and unrealistic scale effects \cite{Marigo2023}.  

To address these shortcomings, cohesive fracture models introduce nonzero forces between crack surfaces, the cohesive forces, making the surface energy density a function of the displacement jump. Originally proposed by Dugdale and Barenblatt \cite{Dugdale1960,Barenblatt1962}, these models incorporate a critical stress and a characteristic length, offering a more realistic fracture representation (\Fig{fig_BCL}-b). Building on this foundation, the variational formulation of fracture, first developed for brittle materials in \cite{Marigo1998,Braides95}, has been extended to cohesive models in \cite{Bourdin2008}.

% \begin{figure}[h!]
%  \centering
%  \small
%   \begin{subfigure}{0.24\linewidth}
%     \centering
%       \textbf{brittle fracture}
%       \\[2mm]
%       \fbox{
%       \includegraphics[page=4, trim=10mm 165mm 220mm 10mm, clip, scale=0.7]{Cohesive_fracture}
%       }
%       \\[2mm]
%       \includegraphics[page=1, trim=10mm 130mm 200mm 0mm, clip, width=\linewidth]{Cohesive_fracture}
% %     \fbox{
% %     }
%     \caption{}
%     \label{figsub_G} 
%   \end{subfigure}
%   \rule[-45mm]{0.2mm}{75mm}
%   \begin{subfigure}{0.72\linewidth}
%     \centering
%       \textbf{cohesive fracture}
%       \\[2mm]
%       \fbox{
%       \includegraphics[page=4, trim=75mm 165mm 100mm 10mm, clip, scale=0.7]{Cohesive_fracture}
%       }
%       \\[2mm]
%       \includegraphics[page=1, trim=140mm 130mm 70mm 0mm, clip, width=0.32\linewidth]{Cohesive_fracture}
%       \includegraphics[page=1, trim=75mm 130mm 135mm 0mm, clip, width=0.32\linewidth]{Cohesive_fracture}
%       \includegraphics[page=1, trim=205mm 130mm 5mm 0mm, clip, width=0.32\linewidth]{Cohesive_fracture}
%     \caption{}
%     \label{figsub_C} 
%   \end{subfigure}
% \caption{Qualitative trends of the surface energy density $\Gf$ and associated cohesive stress $\scristr$ with respect to the crack opening (displacement jump) $\delta$ for brittle and cohesive fracture models: (\subref{figsub_G}) Griffith, (\subref{figsub_C}) Dugdale, Barenblatt and linear cohesive models. $\Gc$ and $\scristr_0$ denote the fracture toughness and the initial critical stress, respectively.}
% \label{fig_BCL}
% \end{figure}

\begin{figure}[h!]
 \centering
 \small
  \includegraphics[width=\linewidth]{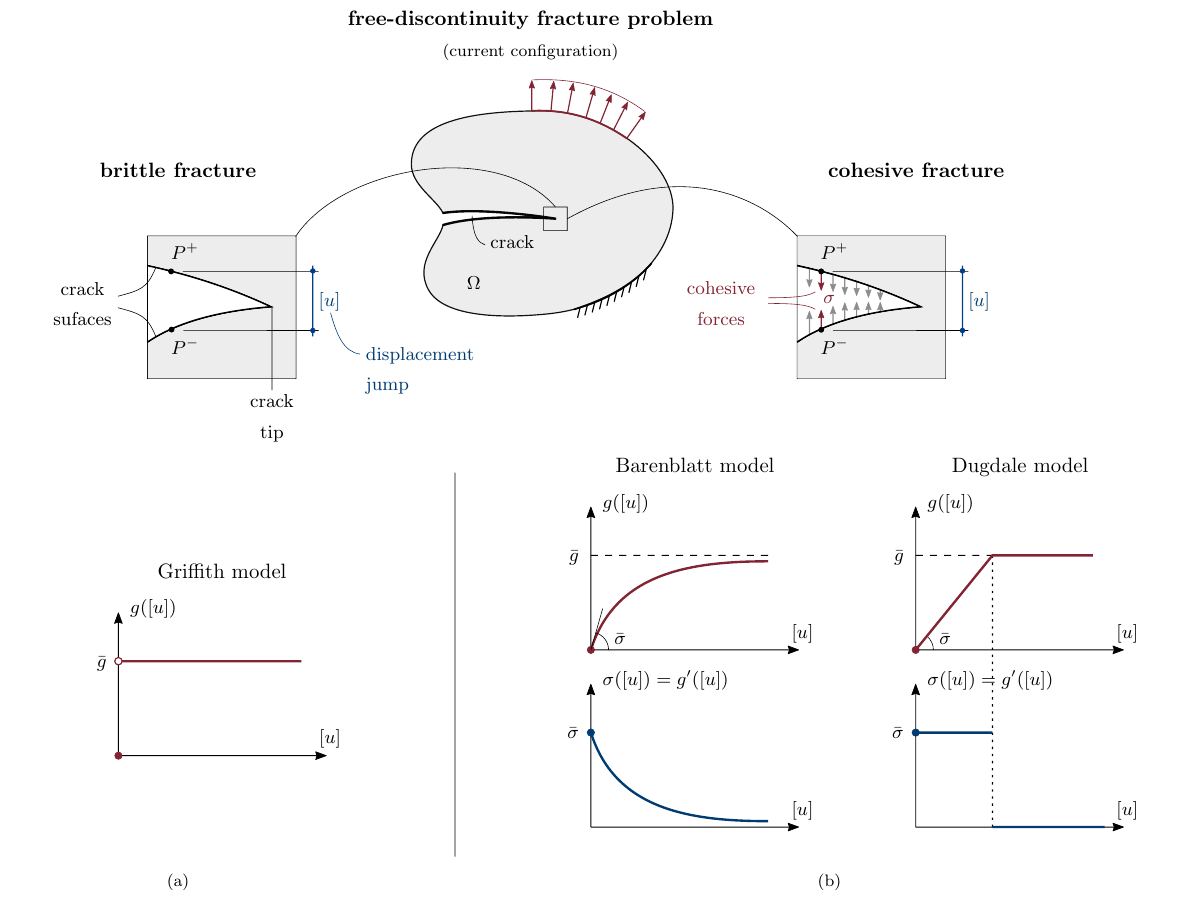}
 \caption{Qualitative trends of the surface energy density $\Gf$ and associated cohesive stress $\scristr$ with respect to the crack opening (displacement jump) $\delta$ for brittle and cohesive fracture models: (a) Griffith, (b) Dugdale and Barenblatt cohesive models. $\Gc$ and $\scristr_0$ denote the fracture toughness and the initial critical stress, respectively.}
 \label{fig_BCL}
\end{figure}

The phase-field regularization of brittle fracture, \cite{Ambrosio1992,Bourdin2000b,Focardi2001}, allows cracks to emerge naturally within the $\Gamma$-convergence framework, \cite{Braides1998}, and enables the numerical simulation of complex fracture processes, making it nowadays a leading approach in fracture mechanics \cite{Bourdin2014,Mesgarnejad2015}. For instance it captures crack nucleation, even without pre-existing singularities, and models complex crack patterns with straightforward numerical methods, often based on alternate minimization schemes.

Phase-field models for brittle fracture inherit a limitation from Griffith's theory: they do not allow independent control of critical strength, fracture toughness, and regularization length. 
As a consequence, the capability to model crack nucleation in both smooth and notched domains \cite{Bourdin2008,Tanne2018,Lopez-Pamies2024} is restricted.
% This limits flexibility in modelling crack nucleation in both smooth and notched domains \cite{Bourdin2008,Tanne2018,Lopez-Pamies2024}.

Cohesive fracture models turned out to being able to overcome this limitation of Griffith's brittle fracture theory, \cite{Marigo2023}.
Building on the work of \cite{Pham2010c,Pham2010a}, \cite{Lorentz2011,Lorentz2011a} developed a standalone gradient-damage (phase-field) model, not derived from a free-discontinuity fracture model, that is able to describe a cohesive fracture behavior. The model has been validated in a one-dimensional uniaxial tension scenario by a closed-form solution and in higher-dimensional problems with numerical simulations of large-scale domains.
Despite its potential, this cohesive phase-field model was initially overlooked by many in the fracture mechanics community.

Later on, \cite{Conti2016} successfully developed a mathematically consistent variational phase-field model that regularizes the free-discontinuity cohesive fracture problem of \cite{Bourdin2008} (see also \cite{ContiFocardiIurlano2022,ContiFocardiIurlanopgrowth,Colasanto2024} for vectorial counterparts). The numerical implementation of this model, done by \cite{Freddi2017}, encountered several challenges, partially overcome by adopting a backtracking algorithm and by further regularizing the degradation function, similarly to what has been done in \cite{Lammen2023,Lammen2025}. Additionally, difficulties in adjusting the elastic degradation function to fit specific cohesive laws, along with the use of a fixed quadratic phase-field dissipation function, likely constrained the model's flexibility, limiting its wider acceptance in the fracture mechanics community.

More impactful progress was made by \cite{Wu2017,Wu2018b}, who, following \cite{Lorentz2011}, proposed a next-generation phase-field cohesive fracture model within the familiar variational framework of gradient-damage models \cite{Pham2010c,Pham2010a}. These models use polynomial crack geometric functions and rational energetic degradation functions, enabling the independent adjustment of critical stress, fracture toughness, and regularization length to match specific softening behaviors. A defining feature of the models proposed by \cite{Conti2016,Wu2017,Wu2018b} is the incorporation of the regularization length into the elastic degradation function, enhancing their ability to describe cohesive fracture failures. Examples of commonly used traction-separation laws in these models include linear, exponential, hyperbolic, and Cornelissen \cite{Cornelissen1986} laws.

Despite the model's flexibility and accuracy, determining how to set the material functions to obtain a specific target softening traction-separation law was still unclear. A breakthrough came from \cite{Feng2021}, who introduced an integral relation that links a single unknown function, defining both the degradation function and phase-field dissipation function, to the desired traction-separation law. In these phase-field models, both the regularization length and mesh size were found to have minimal impact on the overall mechanical response, provided the mesh adequately resolves the regularization length.
One of the main goals of this Part~II paper is exactly to validate analitically and to generalise the model-finding process to match a target cohesive fracture law.

Indeed, having established in the Part~I paper \cite{Alessi2025a} a $\Gamma$-convergence result in a one-dimensional setting for a broad class of phase-field energies for the description of cohesive fracture that encompasses the models proposed in \cite{Conti2016}, \cite{Wu2017,Wu2018b}, \cite{Feng2021}, and \cite{Lammen2023,Lammen2025}, in this second paper, we develop a general and flexible procedure for constructing phase-field models that reproduce prescribed cohesive traction-separation laws
exhibiting either a linear or a superlinear behaviour for small jump amplitudes
(for the analysis of the corresponding vector-valued model see \cite{Colasanto2024}). This is achieved by either fixing the degradation function and determining the damage potential or vice versa.
% It is worth pointing out that our approach rigorously justifies and extends the results of \cite{Feng2021}, enabling the derivation of multiple phase-field models that correspond to the same target cohesive law but exhibit distinct localized phase-field evolutions, as shown by the several examples worked out.

It is worth pointing out that our approach rigorously justifies and extends the results for the linear case of \cite{Feng2021}, enabling the derivation of multiple phase field models that correspond to the same target cohesive law but exhibit distinct localized phase field evolutions, as shown by the several examples worked out.
In addition,  the superlinear case has been shown to be particularly relevant for describing fracture processes in ductile materials by strain-gradient plasticity models~\cite{Fokoua2014}.
Our results, that provide a unified framework for the description of cohesive fracture, are validated in the Part~III paper, \cite{Alessi2025c}, where the mechanical responses of different phase-field models are investigated in depth with a more engineering-oriented perspective. Moreover, the third paper also contains in its introduction a state-of-art timeline that illustrated the key milestones, their interconnections, and how this work fits within the broader variational framework of fracture mechanics.

\subsection{A recap of the phase-field model and $\Gamma$-convergence}\label{ss:recap}

% ------------------------------------------------------------------------------
%\subsection{Setting of the problem}\label{ss:data}

We briefly recall the phase-field model introduced and analyzed in \cite{Alessi2025a} that encompasses at the same
time several cohesive phase-field models present in the literature (see \cite{Conti2016,Wu2017,Wu2018b,Lammen2023,Lammen2025}). With this aim consider functions
$\hatf,\, Q\,,\PotDann:[0,1]\to[0,\infty)$, $f:[0,1)\to[0,\infty)$ and $\varphi:[0,\infty)\to[0,\infty)$ satisfying:
\begin{itemize}
\item[(Hp~$1$)] $\hatf,\, Q\in C^0([0,1],[0,\infty))$ with $\hatf (1)= 1$,
$\hatf^{-1}(\{0\})=Q^{-1}(\{0\})=\{0\}$, $Q$ increasing in a right neighbourhood of the origin,
and such that
\begin{equation}\label{e:f}
[0,1)\ni t\mapsto f(t):=\left( \frac{\hatf(t)}{Q(1-t)}\right)^{\sfrac12}\,,
\end{equation}
is increasing in a right neighbourhood of the origin;
\item[(Hp~$2$)] $\PotDann\in C^0([0,1],[0,\infty))$ such that $\PotDann^{-1}(\{0\})=\{0\}$,
and the following limit exists
\begin{equation}\label{e:FailureS def}
\lim_{t\to 0^+}\left(\frac{\PotDann(t)}{Q(t)}\right)^{\sfrac12}
=:\FailureS\in[0,\infty]\,;
 \end{equation}
\item[(Hp~$3$)] $\varphi\in C^0([0,\infty))$, $\varphi$ is non-decreasing and %bounded with
$\varphi(\infty):= \displaystyle{\lim_{t\to\infty}\varphi(t)}\in(0,\infty)$;
\item[(Hp~$4$)] $\varphi^{-1}(0)=\{0\}$, and
$\varphi$ is (right-)differentiable in $0$ with $\varphi'(0^+)\in(0,\infty)$.
\end{itemize}
Note that $[0,1]\ni t\mapsto f_{\eps}(t)$ is increasing in a right neighbourhood of the origin, as well.

We adopt standard notation for Sobolev, $BV$, $GBV$ spaces as in \cite{AFP}. We refer to the same book for results
related to those function spaces.
Then define the functionals $\Functeps:L^1(\Omega,\R^2)\times\calA(\Omega)\to[0,\infty]$ by
\begin{equation}\label{functeps}
 \Functeps(u,v,A):= \int_A \left( \varphi(\eps f^2(v)) |u'|^2 + \frac{\PotDann(1-v)}{4\eps} + \eps |v'|^2 \right) \dx\,
\end{equation}
if $(u,v)\in H^1(\Omega, \R\times[0,1])$ and $\infty$ otherwise. In particular, in formula \eqref{functeps} above, the functions $[0,1)\mapsto\varphi(\eps f^2(t))$ are extended by continuity to $t=1$ with value $\varphi(\infty)$ for every $\eps>0$.
Now let $\Functu_{\FailureS}:L^1(\Omega)\times \mathcal{A}(\Omega)\to [0,\infty]$ be given by
\begin{equation}\label{e:H}
 \Functu_{\FailureS}(u,A):=
 \begin{cases}
\displaystyle{\int_{A}\hsigmabarra^{**}(|u'|)\dx +(\varphi'(0^+))^{\sfrac12}\FailureS|D^cu|(A) +\int_{J_u\cap A}g(|[u]|)\dd \calH ^0} \; \; & \\
   & \hskip-2cm\textup{if} \; u\in GBV(\Omega) \\
\infty & \hskip-2cm\textup{otherwise} \\
 \end{cases}
\end{equation}
where for $\varsigma\in[0,\infty]$ we define $h_\varsigma:[0,\infty) \to [0,\infty)$ by
\begin{equation}\label{e:hsigma}
  h_\varsigma(t):=\inf_{\tau\in(0,\infty)}
\left\{\varphi\left(\frac1\tau\right)t^2 +\frac{\varsigma^2}{4}\tau\right\}\,,
\end{equation}
(in particular, $h_0(t)=0$ and $h_\infty(t)=\varphi(\infty)t^2$ for every $t\in[0,\infty)$),
$\hsigmabarra^{**}$ denotes the convex envelope of $\hsigmabarra$, and
$g:[0,\infty)\to \mathbb{R}$ is defined by
\begin{equation}\label{e:lagiyo}
  g(s):=\inf_{(\gamma,\beta)\in \calulu _s}
  \int_0^1 \Big(\PotDann(1-\beta)\big(
  \varphi'(0^+)f^2(\beta)|\gamma'|^2
  +|\beta'|^2\big)\Big)^{\sfrac12} \dx\,,
\end{equation}
where $\calulu_s$ is
\begin{align}\label{e:glispaziu}
  \calulu_{s}=\{\gamma,\beta \in H^1((0,1)):\,&
  \gamma(0)=0\,,\gamma(1)=s\,, %\nonumber\\ &
  0\leq \beta \leq 1\,, \beta(0)=\beta(1)= 1\}\,,
\end{align}
%for every $T>0$, with the following convention $\calulu_{s}:=\calulu_{s}(0,1)$.

In \cite[Theorem~1.1]{Alessi2025a} we have established the following $\Gamma$-convergence result
(cf. \cite{Dalmaso1993, Braides1998} for theory and results on $\Gamma$-convergence).
In what follows we adopt the convention $0\cdot\infty=0$. 
\begin{theorem}\label{t:finale}
Assume (Hp~$1$)-(Hp~$4$), and \eqref{e:FailureS def} hold with $\FailureS\in(0,\infty]$.
Let $\Functeps$ be the functional defined in \eqref{functeps}.
Then, for all $(u,v)\in L^1(\Omega,\R^2)$
\[
\Gamma({L^1})\text{-}\lim_{\eps\to0}\Functeps(u,v)=
%\Gamma({L^1})\text{-}\lim_{\eps\to0}\widetilde{\Functeps}(u,v)=
F_{\,\FailureS}(u,v)\,,
\]
where, %for every $\varsigma>0$,
$F_{\,\FailureS}:L^1(\Omega,\R^2)\times\calA(\Omega)\to[0,\infty]$ is defined by
\begin{equation}\label{F0}
F_{\,\FailureS}(u,v):=
\begin{cases}
\Functu_{\,\FailureS}(u) & \textup{if $v=1$ $\mathcal{L}^1$-a.e. on $\Omega$}\cr
\infty & \textup{otherwise}.
\end{cases}\,
\end{equation}
\end{theorem}
In particular, if $\Functu_\infty(u)<\infty$ above, then $|D^cu|(\Omega)=0$ (cf. \eqref{e:H}), and thus we conclude that $u\in GSBV(\Omega)$ with
$u'\in L^2(\Omega)$

We recall that changing the scaling in the degradation function $\varphi(\eps f(t))$ to $\varphi(\gamma_\eps f(t))$,
with $\sfrac\eps{\gamma_\eps}\to0 $ as $\eps\to0^+$, then the corresponding functionals \eqref{functeps} $\Gamma$-converge
to a brittle type functional (cf. \cite[Theorem~1.2]{Alessi2025a}).

\subsection{Contents of this Paper}
% ------------------------------------------------------------------------------

The great generality in which Theorem~\ref{t:finale} is proven is instrumental to address the issue of constructing 
the phase-field model in order to obtain an assigned cohesive law $g$. In Section~\ref{s:sezione 5.3} we take advantage of Theorem~\ref{t:finale} to solve the latter problem for a large class of concave functions with either a linear or a superlinear behaviour for small jump amplitudes, including several examples which are typically used in the engineering community for numerical simulations (cf. Section~\ref{ss:esempi g esplicite}).
The latter problem has already been discussed from a numerical and engineering perspective in \cite{Wu2017,Wu2018b,Feng2021,Lammen2025,Wu2025}.

Section~\ref{s:assegnare g} is devoted to provide an alternative characterization of the surface energy densities
in the phase-field model. Our approach is different to the one introduced in \cite{BCI,BI23} (cf. the comments before Proposition~\ref{p:cambio variabile in v old} for a more detailed comparison).
One of the main ideas we introduce is rewriting in a convenient way the minimization formula that defines the surface
energy density $g(s)$ for every jump amplitude $s\geq0$ (cf. \eqref{e:lagiyo} in Theorem~\ref{t:finale}).
%(cf. \eqref{e:la g esplicita}).
Indeed, the latter entails the solution of a vector-valued, one-dimensional variational problem with Dirichlet boundary conditions. 
Via a change of variable we are able to reduce it to a minimum problem for a family of scalar and one-dimensional functionals $G_{m,s}$, 
defined in \eqref{e:funzionale Gms}, depending on the deformation gradient, parametrized by the 
minimum value $m$ of the phase-field variable and finite on $W^{1,1}$ functions satisfying a suitable Dirichlet boundary condition depending on $s$
(cf. Proposition~\ref{p:cambio variabile in v old}). The new minimum problem is obtained by minimizing first 
each functional $G_{m,s}$ on $L^1$, and then minimizing the corresponding minimimum value $\mathfrak{G}_s(m)$ on the interval
$(0,1)$
(cf. \eqref{e:doppio inf rilassato G} and \eqref{e:mathfrakGs}),
namely $g(s)=\inf_{m\in(0,1)}\mathfrak{G}_s(m)$ where $\mathfrak{G}_s(m)=\inf_{L^1} G_{m,s}$.

Another key step in our approach is the variational analysis of the energies $G_{m,s}$ introduced in \eqref{e:funzionale Gms}.
%(cf. Section~\ref{ss:equivalent characterization}).
In particular, for each $m$ we determine the relaxed functional with respect to the $L^1$ topology of 
$G_{m,s}$, which turns out to be finite on $BV$, and to have a unique minimizer that it is either 
$W^{1,1}$ or $SBV\setminus W^{1,1}$ with the singular part of the distributional derivative concentrated 
in the point $t=m$ in the second instance (cf. Theorem~\ref{t:regolarità minimo old} and Corollary~\ref{c:regolarità minimo old}).
The analysis of the function $\mathfrak{G}_s$ exploits an auxiliary key function $\Phi$, defined in \eqref{e:la psi piccola}, determining its monotonicity properties. 
More precisely, $\mathfrak{G}_s$ turns out to be increasing on intervals for which the minimizer $w_{m,s}$ of the relaxation of $G_{m,s}$ is $SBV\setminus W^{1,1}$ and decreasing otherwise (for the precise statement see Lemma~\ref{l:derivate energia}). 
In addition, the regularity of the minimizer $w_{m,s}$ depends on whether $m$ belongs or not to the $s$ sublevel set of $\Phi$. Given this, we are able to show that $g(s)$ is reduced either to $\inf\{\mathfrak{G}_s(m):\,\Phi(m)=s\}$ or to its behaviour for $m=0$ and $m=1$, for every $s$ in the range of $\Phi$ (for the precise statement see Proposition~\ref{l:Gs reduction}).
In particular, in case $\Phi$ is strictly decreasing then $g$ turns out to be concave and of class $C^1$,
with derivative expressed exactly in terms of $\Phi^{-1}(s)$ (cf. Theorems~\ref{t:teo la g esplicita}). 
In turn, this information is used in Theorems~\ref{t:hatf fixed} and \ref{t:potdann fixed} for the linear case,
and in Theorems~\ref{t:potdann fixed infinito} and \ref{t:hatf fixed infinito} in the superlinear one,
to define a phase-field model choosing either the damage potential and fixing the degradation
function in the former or vice versa in the latter, with corresponding $\Gamma$-limit that has surface
energy density exactly equal to the assigned function $g$. With this aim we are led to solve the Abel integral equation corresponding to the equality $g(s)=\mathfrak{G}_s(\Phi^{-1}(s))$ (cf. \eqref{e:Abel equation}).
A similar use of Abel integral equation is also proposed in \cite[Section~3.2]{Feng2021} for the heuristic calibration of the damage potential in the phase-field model for numerical convergence purposes.
Instead, in the linear case if $\Phi$ is non-decreasing the function $g$ always corresponds to the Dugdale cohesive law
(cf. Theorems~\ref{t:teo la g esplicita Dugdale}).

% ______________________________________________________________________________
\section{Analysis of the surface energy densities arising in the phase-field approximation}\label{s:assegnare g}

In this section we study in depth the surface energy densities $g$ obtained via the phase-field approximation
in Theorem~\ref{t:finale}. For solving the reconstruction problem it is key to provide an equivalent and more
maneageable characterization with respect to its definition in formula \eqref{e:lagiyo}.

With this aim, throughout the whole section $\hatf$, $Q$, $\PotDann$ and $\varphi$ are assumed to satisfy (Hp~$1$)-(Hp~$4$)
even if not explicitly mentioned.

It is convenient to recall several qualitative properties of $g$ that follows directly from \eqref{e:lagiyo} established in \cite[Section~2.3]{Alessi2025a}. With this aim we introduce the function $\Psi:[0,1]\to[0,\infty)$ given by
\begin{equation}\label{e:Psi}
\Psi(t):=\int_0^t\PotDann^{\sfrac12}(1-\tau)d\tau\,.
\end{equation}
We start with the case in which $g$ is linear for small jump amplitudes.
\begin{proposition}\label{p:lepropdig}
Under the assumptions of Theorem~\ref{t:finale} with $\FailureS\in(0,\infty)$, the function $g$ defined in \eqref{e:lagiyo} enjoys the following properties:
\begin{itemize}
\item[(i)] $g(0)=0$, $g$ is non-decreasing, and subadditive;

\item[(ii)] $g(s)\leq (\varphi'(0^+))^{\sfrac12}\,\FailureS s\wedge 2\Psi(1)$, $g$ is
Lipschitz continuous with Lipschitz constant equal to $(\varphi'(0^+))^{\sfrac12}\,\FailureS$;

\item[(iii)] the ensuing limit exists and
\begin{equation}\label{e:Lafigrandeinfi}
\lim_{s\to \infty} g(s)=2\Psi(1);
\end{equation}

\item[(iv)] the ensuing limit exists and
\begin{equation}\label{e:strenght}
\lim_{s\to 0}\frac{g(s)}{s}=(\varphi'(0^+))^{\sfrac12}\,\FailureS.
\end{equation}

\end{itemize}
\end{proposition}
We then consider the superlinear setting.
\begin{proposition}\label{p:lepropdig infty}
Under the assumptions of Theorem~\ref{t:finale} with $\FailureS=\infty$,
the function $g$ defined in \eqref{e:lagiyo} enjoys the following properties:
\begin{itemize}
\item[(i)] $g(0)=0$,  $g$ is non-decreasing, and subadditive;
\item[(ii)] $g\in C^0([0,\infty))$, $0\leq g(s)\leq \widehat{g}(s)$ for every $s\geq 0$, where
\begin{equation}\label{e:g zero def}
    \widehat{g}(s):=\inf_{x\in(0,1]}\left\{2(\Psi(1)-\Psi(1-x))+\Big(\varphi'(0^+)\hatf(1-x)\frac{\PotDann(x)}{Q(x)}\Big)^{\sfrac12}\,s\right\}\,,
\end{equation}
$\widehat{g}\in C^0([0,\infty))$ is concave, non-decreasing, $\widehat{g}(s)\leq 2\Psi(1)$, and
\begin{equation}\label{e:g zero behaviour in 0}
\lim_{s\to 0^+}\frac{\widehat{g}(s)}s=\infty;
\end{equation}

\item[(iii)]
\begin{equation*}%\label{e:Lafigrandeinfi}
\lim_{s\to \infty} g(s)=2\Psi(1);
\end{equation*}

\item[(iv)]
\begin{equation}\label{e:g behaviour in 0}
2^{-\sfrac12}\leq\liminf_{s\to 0^+}\frac{g(s)}{\widehat{g}(s)}\leq
\limsup_{s\to 0^+}\frac{g(s)}{\widehat{g}(s)}\leq 1\,.
\end{equation}

\end{itemize}
\end{proposition}

For the sake of notational simplicity we also assume that the function $\varphi$ in the definition of the degradation
function obeys $\varphi'(0^+)=1$ (clearly, this restriction can be easily dropped).
Finally, we also assume that $Q$, $\hatf$ and $\PotDann$ satify
\begin{itemize}
\item[(Hp~$5$)] $\hatk:[0,1)\to [0,\infty)$ is strictly increasing, where
\[
\hatk(t):=\frac{\omega(1-t)}{Q(1-t)}\hatf(t)\,.
\]
\end{itemize}
Note that $\hatk(t)=\PotDann(1-t)f(t)$ (cf. \eqref{e:f}).
Necessarily, $\hatk\in C^0([0,1))$, with $\hatk(0)=0$ and $\hatk (1^-) = \FailureS^2\in(0,\infty]$ (cf. \eqref{e:FailureS def} and recall that $\hatf(1)=1$). For some results we will need to strengthen (Hp~$5$) (cf. assumption (Hp~$5'$) below).

With this notation at hand, formula \eqref{e:lagiyo} rewrites as  
\begin{equation}\label{e:definizione g assegnata}
    g(s)=\inf_{(u,v)\in \calulu_s} \mathcal{G}(u,v)
\end{equation}
for every $s\in [0,\infty)$, where $\calulu_s$ is defined in \eqref{e:glispaziu}
(see also $\calulu_{s,1}$ below), and
$\mathcal{G}:W^{1,1}((0,1);\R^2)\to[0,\infty)$ is defined by
\begin{equation}\label{e:def G}
\mathcal{G}(u,v):=    \int_0^1 \Big(\hatk(v)|u'|^2+\PotDann(1-v)|v'|^2\Big)^{\sfrac12}\dx\,.
\end{equation}
We remark that the infimum defining $g$ can be equivalently taken on $\calulu_{s,1}$, where %for every $p\in[1,\infty]$
\begin{align}\label{e:glispaziu bis}
  \calulu_{s,1}:=\{u\in W^{1,1}((0,1)),&\,v\in H^1((0,1)): %\nonumber\\ &
  u(0)=0\,,u(1)=s\,,\,
  0\leq v \leq 1\,, v(0)=v(1)= 1\}\,,
\end{align}
as the functional $\mathcal{G}$ is continuous in the $W^{1,1}$ topology.

\subsection{An equivalent characterization of $g$}

In this section we introduce a different minimization problem defining $g$ 
which will be more suitable for our purposes.
With this aim, in order to reduce first the space of competitors for $v$ to those satisfying suitable monotonicity properties we first prove a technical lemma.
\begin{lemma}\label{l:mantieni regolarità}
Let %$a,b\in \R$, and 
$v\in W^{1,p}((0,1))\cap C^0([0,1])$ with $p\in [1,\infty]$, and 
for every $x\in[0,1]$ set
\begin{equation*}
 T_0 v(x):=\inf_{t\in [0,x]}v(t),\quad\text{and}\quad  
 T_1 v(x):=\inf_{t\in [x,1]}v(t)\,.
\end{equation*}
Then, 
\begin{itemize} 
\item[(i)] $T_0v,\,T_1v\leq v$ on $[0,1]$, $T_0v$ is non-increasing with $T_0v(0)=v(0)$, $T_1v$ is non-decreasing with $T_1v(1)=v(1)$;
\item[(ii)] $T_0v,\,T_1v\in W^{1,p}((0,1))\cap C^0([0,1])$, with $|(T_0v)'|,\,|(T_1v)'|\leq|v'|$ $\calL^1$ a.e on $(0,1)$;
\item[(iii)] $(T_iv)'=0$ $\calL^1$ a.e on $\{T_iv\neq v\}$, $i\in\{0,1\}$.
\end{itemize}
\end{lemma}
\begin{proof}
We give the proof of the statements for $T_0v$, that for 
$T_1v$ being analogous.

The monotonicity of $T_0v$ follows straightforwardly from the very definition, as well as the fact that $T_0v\leq v$.

Let $x_1,x_2\in (0,1)$, with $x_1<x_2$, if $T_0v(x_2)<T_0v(x_1)$
there exist $y_1 \in [0,x_1]$ and $y_2\in (x_1,x_2]$ such that 
\begin{equation*}
    v(x_1)\geq v(y_1)=T_0v(x_1)> T_0v(x_2)=v(y_2)\,.
\end{equation*}
In particular, it follows that 
\[
|T_0v(x_1)-T_0v(x_2)|\leq |v(x_1)-v(y_2)|
\leq\int_{x_1}^{x_2}|v'(t)|\dd t\,.
\]
Therefore, $T_0v$ is absolutely continuous and the statements in item (ii) holds. 

Finally, note that if $T_0v(x)=v(y)$ for some $y\in[0,x]$, then
$T_0v(t)= v(y)$ for every $t\in[y,x]$ by item (i).
Thus, if $I=(x_1,x_2)$ is a connected component of the open set 
$\{T_0v<v\}$ (by continuity of both $v$ and $T_0v$), then $T_0v(x)=v(x_1)$ for every $x\in I$,
and finally $T_0v(x_2)=v(x_1)$ by continuity of $T_0v$.
\end{proof}
\begin{proposition}\label{p:competitori monotoni}
Assume (Hp~$1$)-(Hp~$5$). For every $s\in [0,\infty)$  
\begin{equation}\label{e:calulu s1}
      g(s)=\inf_{(u,v)\in \calulu_{s,1}}\mathcal{G}(u,v).
\end{equation}

Moreover, for every $s\in (0,\infty)$, $\eta>0$, and $(u,v)\in \calulu_{s,1}$ 
there exists 
$\overline{v}\in H^1((0,1))$ such that $\overline{v}(0)=\overline{v}(1)=1$, 
$\overline{v}$ is strictly positive, strictly decreasing on $[0,x_0]$, strictly 
increasing on $[x_0,1]$, for some $x_0\in (0,1)$, and with Lipschitz inverse when 
restricted both to $[0,x_0]$ and to $[x_0,1]$, and such that
\begin{equation}\label{e:Goverline uv vs Guv}
    \mathcal{G}(u,\overline{v}) \leq \mathcal{G}(u,v)+\eta\,.
\end{equation}
\end{proposition}
\begin{proof}
 Fix $(u,v)\in \calulu_{s,1}$ such that $\mathcal{G}(u,v)<\infty$. Let $(z_k)$ be an increasing sequence of positive functions in $L^{\infty}((0,1))$ such that $z_k(x)\to |u'(x)|$ for $\calL^1$-a.e. $x\in (0,1)$. In particular setting 
 \begin{equation*}
     w_k:=\frac{s}{\int_0^1 z_k \dx}z_k
 \end{equation*}
 we have that $(u_k,v)\in \calulu_{s,1}$ with $u_k(t):=\int^t_0 w_k \dx\in W^{1,\infty}((0,1))$ and 
 \begin{equation*}
     \lim_{k\to \infty}\mathcal{G}(u_k,v)\leq \mathcal{G}(u,v)\,,
 \end{equation*}
so that \eqref{e:calulu s1} easily follows.

Let $v$ be as in the statement above, without loss of generality we assume that $v\in C^0([0,1])$. 
Let $v_k:=v \lor \frac{1}{k}$ then $v_k(x) \to v(x)$, $\PotDann(1-v_k(x))|v'_k(x)|^2=0$ $\calL^1$-a.e. on $\{v\leq\frac1k\}$, and $v_k=v$ otherwise, so that 
$\PotDann(1-v_k(x))|v'_k(x)|^2\leq\PotDann(1-v(x))|v'(x)|^2$ $\calL^1$-a.e. $(0,1)$. In particular, 
$\mathcal{G}(u,v_k)\to\mathcal{G}(u,v)$ for $k\to\infty$, 
as for every $k\geq2$   
 \begin{align*}
  &   \mathcal{G}(u,v_k)\leq \int_0^1 \Big(\hatk(v_k) |u'|^2 + \PotDann(1-v)|v'|^2\Big)^{\sfrac12} \dx \\
  & = \int_{\{v\leq \frac{1}{2}\}} \Big(\hatk(v_k) |u'|^2 + \PotDann(1-v)|v'|^2\Big)^{\sfrac12} \dx
  + \int_{\{v> \frac{1}{2}\}} \Big(\hatk(v) |u'|^2 + \PotDann(1-v)|v'|^2\Big)^{\sfrac12} \dx.
 \end{align*}
Hence, without loss of generality we can suppose $v$ strictly positive.
Let then $x_0\in (0,1)$ be a global minimum point of $v$, and 
define $\hat{v}:[0,1]\to [0,1]$ by
\[
\hat v(x):=T_0v(x)\chi_{[0,x_0]}(x)
+T_1v(x)\chi_{[x_0,1]}(x)\,.
\]
Lemma \ref{l:mantieni regolarità} yields that 
  $\hat v\in H^1((0,1))$, $\hat v(0)=\hat v(1)=1$, $\hat v$ is non-increasing on $[0,x_0]$ and non-decreasing on $[x_0,1]$.
  Moreover, by item (iii) there
  \begin{equation*}
 \mathcal{G}(u,\hat{v})\leq \mathcal{G}(u,v)\,
  \end{equation*}
  as $\hatk$ is non-decreasing and $\hat v'=0$ $\calL^1$-a.e. on $\{\hat v \neq v\}$. Next, consider the piecewise affine function
  \begin{equation*}
       \zeta(x):=
       \begin{cases}
           \frac{x}{x_0}\; & \textup{if $0\leq x \leq x_0$} \\
           -\frac{x-x_0}{1-x_{0}}+1 \; & \textup{if $x_0\leq x \leq 1$}, \\
       \end{cases}
   \end{equation*}
observe that $\hat v-\delta \zeta$ is strictly positive for $\delta$ small enough, and converges to $\hat v$ monotonically as $\delta\to0$, so that $\mathcal{G}(u,\hat v-\delta \zeta)\to \mathcal{G}(u,\hat v)$ as $\delta \to 0$. 
Therefore, \eqref{e:Goverline uv vs Guv} is satisfied. Moreover, the restrictions of $\hat v-\delta \zeta$ to $[0,x_0]$ and $[x_0,1]$ are invertible with, respectively, derivative less than or equal to $-\delta$ and bigger than or equal to $\delta$, and thus they are Lipschitz continuos.
\end{proof}
We are now ready to give a new characterization of $g$ as a minimimum problem obtained by first minimizing over the variable $u$ (belonging to  a suitable admissible set) a family of functionals labeled through the minimimum value of $v$, and then minimizing on the latter. 
This is alternative to the approach in \cite{BCI,BI23} where $g$ is
expressed as a minimum problem of a single functional depending on $v$ loosely speaking using $u$ to change variables. With this aim it is convenient to denote by $\zeta_{m,s}$ the unique affine function with $\zeta_{m,s}(2m-1)=0$ and $\zeta_{m,s}(1)=s$. 

\begin{proposition}\label{p:cambio variabile in v old}
Assume (Hp~$1$)-(Hp~$5$). For every $m\in (0,1)$ let $G_m:W^{1,1}((2m-1,1))\to [0,\infty]$ be given by
\begin{equation}\label{e:funzionale con parametro m}
    G_m(w):=\int_{2m-1}^1\Big(\hatk_{m}(t)|w'|^2+\PotDann_{m}(1-t)\Big)^{\sfrac12}\dt
\end{equation}
where $\hatk_m$ and $\PotDann_m$ are, respectively, the even extensions to $[2m-1,m]$ with respect to $t=m$ of $\hatk$ and $\PotDann(1-\cdot)$ restricted to $[m,1]$. 
Then, for every $s\in[0,\infty)$
\begin{align}\label{e:espressione g doppio inf old}
     g(s)&=\inf_{m\in (0,1)}\inf \{G_m(w) \; : \; w\in \zeta_{m,s}+W^{1,1}_0((2m-1,1)) \}\,.
     \end{align}
\end{proposition}
\begin{proof}
Let $s\in (0,\infty)$ and $m\in (0,1)$. In view of Proposition~\ref{p:competitori monotoni} we may take $v\in H^1((0,1))\cap C^0([0,1])$ such that $0\leq v\leq 1$, $v(0)=v(1)=1$,
$v(x_0)=\min_{[0,1]}v=m$, for some $x_0\in (0,1)$,
and $v|_{[0,x_0]}$ and $v|_{[x_0,1]}$ invertible with
Lipschitz continuous inverses. For every 
$u\in \zeta_{\sfrac{1}{2},s}+W^{1,1}_0((0,1))$, define
$w\in \zeta_{m,s}+W^{1,1}_0((2m-1,1))$ by 
 \begin{equation}\label{e:cambio di variabile per la w nel nuovo funzionale}
     w(t)=\begin{cases}
         u\circ (v|_{[0,x_0]})^{-1}(2m-t) \; & t\in[2m-1,m] \\
         u\circ (v|_{[x_0,1]})^{-1}(t) \; & t\in[m,1].
     \end{cases}
 \end{equation}
By changing variable on the intervals $[0,x_0]$ and $[x_0,1]$ by means of the one-dimensional area formula \cite[Theorem~3.4.6]{AT04} we obtain 
the equality $\mathcal{G}(u,v)=G_m(w)$.
 Vice versa, being $v$ absolutely continuous, given $w\in\zeta_{m,s}+W^{1,1}_0((2m-1,1))$ one can define $u\in\zeta_{\sfrac{1}{2},s}+W^{1,1}_0((0,1))$ via \eqref{e:cambio di variabile per la w nel nuovo funzionale} in a way that $\mathcal{G}(u,v)=G_m(w)$. In particular, by equality \eqref{e:calulu s1} in Proposition~\ref{p:competitori monotoni}
 we can conclude that \eqref{e:espressione g doppio inf old} holds.
\end{proof}

\subsection{Analysis of the equivalent characterization}\label{ss:equivalent characterization}

In this section we study the new characterization of $g$ provided in Proposition~\ref{p:cambio variabile in v old}.
With this aim, we recall first a relaxation result (cf. \cite[Section~2]{GMS79}, \cite{BB90}).
We recall that we have adopted the convention $0\cdot\infty=0$.
\begin{proposition}\label{p:prop. rilassamento G old}
Assume (Hp~$1$)-(Hp~$5$). Let $m\in (0,1)$,  $s\in (0,\infty)$, and $G_{m,s}:L^1((2m-1,1))\to [0,\infty]$ be the functional given by
\begin{equation}\label{e:funzionale Gms}
    G_{m,s}(w):=\begin{cases}
        G_{m}(w) \; & \textup{if $w\in \zeta_{m,s}+W^{1,1}_0((2m-1,1)),$} \\
        \infty \; &\textup{otherwise},
    \end{cases}
\end{equation}
where $G_m$ is defined in \eqref{e:funzionale con parametro m}.
Then the $L^1$-lower semicontinuous envelope of $G_{m,s}$ is given by
\begin{align}\label{e:rilassato old}
   sc^-G_{m,s}(w)=&\int_{2m-1}^1\Big(\hatk_{m}(t)|w'|^2+\PotDann_{m}(1-t)\Big)^{\sfrac12}\dt
   + \int_{2m-1}^1\hatk_{m}^{\sfrac12}(t)\dd|D^sw| \nonumber \\ &
+ \hatk^{\sfrac12}(1^-)|w((2m-1)^+)|+\hatk^{\sfrac12}(1^-)|w(1^-)-s|, 
\end{align}
for every $w\in BV((2m-1,1))$, and $+\infty$ otherwise on $L^1((2m-1,1))$.
In particular, $w((2m-1)^+)=w(1^-)-s=0$ if $\hatk(1^-)=\infty$. 

Moreover, the following holds
\begin{equation}\label{e:doppio inf rilassato G}
    g(s)=\inf_{m\in (0,1)}\inf_{w\in L^1((2m-1,1))}sc^{-}G_{m,s}(w).\,
\end{equation}
\end{proposition}
\begin{proof}
If $\hatk(1^-)<\infty$ the claim is a direct consequence of \cite[Section~2]{GMS79}. 

Instead, if $\hatk(1^-)=\infty$, 
we first prove a lower bound for the relaxed functional by approximation. 
Indeed, fix $s\in (0,\infty)$, $m\in (0,1)$ and denote by $\mathscr{G}_{m,s}$ the functional on the right hand side of \eqref{e:rilassato old}. 
Let $\hatk^{(j)}(t):=\hatk(t)\wedge j$, $t\in[0,1]$, and apply the equality in \eqref{e:rilassato old} to the corresponding functional in \eqref{e:funzionale Gms} obtained by substituting $\hatk_m$ with $\hatk^{(j)}_m$. By monotone convergence it is then easy to conclude that $sc^-G_{m,s}(w)\geq\mathscr{G}_{m,s}(w)$
for every $w\in BV((2m-1,1))$ with $w(2m-1)=0$ and $w(1)=s$, and that $sc^-G_{m,s}(w)=\infty$ otherwise. 

For the reverse inequality, first choose $w\in BV((2m-1,1))$ with $w\equiv 0$ on 
$(2m-1,2m-1+\lambda)$ and $w\equiv s$ on $(1-\lambda,1)$, for some $\lambda$ sufficiently small. Thanks to the result in the case $\hatk(1^-)<\infty$, 
which actually holds for every interval $(a,b)\subset(0,1)$, 
there exists e sequence $w_j\in H^1((2m-1+\lambda,1-\lambda))$ such that $w_j\to w$ in $L^1((2m-1+\lambda,1-\lambda))$ as $j\to \infty$ and 
\begin{align}\label{e:energia wj lambda}
    \lim_{j\to \infty} &\int_{2m-1+\lambda}^{1-\lambda}
    \Big(\hatk_{m}(t)|w_j'|^2+\PotDann_{m}(1-t)\Big)^{\sfrac12}\dt\notag\\ &
    =\int_{2m-1+\lambda}^{1-\lambda}\Big(\hatk_{m}(t)|w'|^2+\PotDann_{m}(1-t)\Big)^{\sfrac12}\dt
   + \int_{2m-1+\lambda}^{1-\lambda}\hatk_{m}^{\sfrac12}(t)\dd|D^sw|\,.
\end{align}
Clearly, extending $w_j$ to $0$ on $(2m-1,2m-1+\lambda)$, and to $s$ on $(1-\lambda,1)$ implies the $L^1((2m-1,1))$ convergence of the extensions, still denoted by $w_j$ with a slight abuse of notation, to $w$. Moreover, using \eqref{e:energia wj lambda} 
a simple calculation shows that
\begin{equation*}
  sc^-G_{m,s}(w)\leq  \lim_{j\to \infty}G_{m,s}(w_j)= \mathscr{G}_{m,s}(w)\,,
\end{equation*}
and thus the identity in \eqref{e:rilassato old} follows for $w$ as above. 
To establish the latter in general, let $w\in BV((2m-1,1))$ with $w(2m-1)=0$, $w(1)=s$ and
$\mathscr{G}_{m,s}(w)<\infty$, otherwise the proof is completed by the lower bound. 
For every $\lambda$ sufficiently small let $\varphi_\lambda:[2m-1,1]\to[2m-1+\lambda,1-\lambda]
$ given by $\varphi_\lambda(t):=m_\lambda(t-m)+m$, with $m_\lambda:=\frac{1-m-\delta}{1-m}$.
Then define the function $w_\lambda\in BV((2m-1,1))$ by 
\begin{equation}
    w_\lambda(\tau):=
    \begin{cases}
        0 \; &\textup{$\tau\in[2m-1,2m-1+\lambda]$} \\
        w(\varphi_\lambda^{-1}(\tau)) \; &\textup{$\tau\in[2m-1+\lambda,1-\lambda]$} \\
        s  \; &\textup{$\tau\in[1-\lambda,1]$}
    \end{cases}.
\end{equation}
Note that $w_\lambda\to w$ in $L^1((2m-1,1))$ as $\lambda \to 0$, 
and moreover that $D w_\lambda%|\res (2m-1+\lambda,1-\lambda)
=(\varphi_\lambda)_\#(D w)$ (the push-forward measure through $\varphi_\lambda$). 
Thus, we compute as follows
\begin{align*}
sc^-G_{m,s}&(w_\lambda) \\ =&\int_{2m-1+\lambda}^{1-\lambda}\Big(\hatk_{m}(\tau)|w_\lambda'|^2+\PotDann_{m}(1-\tau)\Big)^{\sfrac12}\dd \tau
   + \int_{2m-1-\lambda}^{1-\lambda}\hatk_{m}^{\sfrac12}(\tau)\dd|D^sw_\lambda| \\ 
   & + \int_{2m-1}^{2m-1+\lambda}\PotDann_m^{\sfrac12}(1-\tau) \dd\tau + \int_{1-\lambda}^{1}\PotDann_m^{\sfrac12}(1-t) \dd\tau \\ =&\int_{2m-1}^{1}\Big(\hatk_{m}(\varphi_\lambda(t))
   |w'|^2+\PotDann_{m}(1-\varphi_\lambda(t))\, 
    m_\lambda^{2}\Big)^{\sfrac12}\,\dt \\ &
   + \int_{2m-1}^{1}\hatk_{m}^{\sfrac12}(\varphi_\lambda(t))\dd|D^sw|   
   %+ \int_{2m-1}^{2m-1+\lambda}\sqrt{\PotDann(1-t)} \dt 
   + 2\int_{1-\lambda}^{1}\PotDann^{\sfrac12}(1-t) \dt\\
   \leq&\int_{2m-1}^{1}\Big(\hatk_{m}(t)
   |w'|^2+\PotDann_{m}(1-\varphi_\lambda(t))\Big)^{\sfrac12}\,\dt \\ &
   + \int_{2m-1}^{1}\hatk_{m}^{\sfrac12}(t)\dd|D^sw|   
   %+ \int_{2m-1}^{2m-1+\lambda}\sqrt{\PotDann(1-t)} \dt 
   + 2\int_{1-\lambda}^{1}\PotDann^{\sfrac12}(1-t) \dt\,,
\end{align*}
where in the last inequality we have used that $m_\lambda<1$, and that $\hatk_{m}(\varphi_\lambda(t))\leq \hatk_m(t)$ on $[2m-1,1]$ by taking into account that $\hatk_m$ is decreasing on $[2m-1,m]$ together with the inequality $\varphi_\lambda(t)\geq t$ on the same interval, 
and increasing on $[m,1]$ and the inequality $\varphi_\lambda(t)\leq t$ on the same interval. 
The conclusion the follows at once by Dominated Convergence Theorem and then 
$L^1((2m-1,1))$ lower semicontinuity of $sc^-G_{m,s}$.
\end{proof}

Next, we investigate several properties of the minimizers of $G_{m,s}$ at $m\in(0,1)$ and $s\in(0,\infty)$ fixed.
\begin{theorem}\label{t:regolarità minimo old}
Assume (Hp~$1$)-(Hp~$5$).
For every $m\in (0,1)$ and $s\in (0,\infty)$, let $G_{m,s}$ be defined in \eqref{e:funzionale Gms}.     
Then there exists a function $w_{m,s}$ such that
    \begin{equation}\label{e:def wms}
        sc^{-}G_{m,s}(w_{m,s})=\inf_{w\in L^1((2m-1,1))}sc^{-}G_{m,s}(w).
    \end{equation}
Moreover, $w_{m,s}$ is the unique minimizer of $sc^{-}G_{m,s}$ on $L^1((2m-1,1))$, $w_{m,s}\in SBV((2m-1,1))$,
is strictly increasing and odd symmetric with respect to $t=m$, $w_{m,s}(2m-1)=0$, $w_{m,s}(1)=s$, and
for (a unique) $\lambda_{m,s}\in (0,\hatk^{\sfrac12}(m)]$
\begin{equation}\label{e:el}
             w_{m,s}'(t)=\left(\frac{\lambda_{m,s}^2\PotDann_m(1-t)}{\hatk_m(t)(\hatk_m(t)-\lambda^2_{m,s})}\right)^{\sfrac12}\,
\end{equation}
for $\calL^{1}$-a.e $t\in (2m-1,1)$, and either
    \begin{itemize}
\item[(1)] $w_{m,s}\in W^{1,1}((2m-1,1))$, with $w_{m,s}(m)=\frac s2$;
            \end{itemize}
or
\begin{itemize}
    \item[(2)] $w_{m,s}\in SBV((2m-1,1))$, with $D^jw_{m,s}$ concentrated on $t=m$, and $\lambda_{m,s}=\hatk^{\sfrac12}(m)$.
        \end{itemize}
\end{theorem}

\begin{proof} \noindent{\bf Step 1. Existence of a minimizer.}  
We note that $\inf_{[2m-1,1]}\hatk_m=\hatk_m(m)>0$ thanks to the condition $m>0$. Therefore, 
$sc^{-}G_{m,s}$ is a coercive functional on $BV((2m-1,1))$  (cf. \eqref{e:rilassato old}).
The existence of a minimizer $w_{m,s}$ of $sc^{-}G_{m,s}$ follows from the 
Direct Method of the Calculus of Variations and the BV compactness theorem.
\medskip

\noindent{\bf Step 2. Regularity of minimizers.} 
We prove that any minimizer $w_{m,s}$ attains the boundary data
and the singular part of its distributional derivative is concentrated in $t=m$. 
With this aim the crucial observation is that the function $\hatk_m$ assumes its unique minimum value in $t=m$.
Then, consider
\[
\hat w(t):=\int_{2m-1}^t w_{m,s}'(\tau)d\tau+D^sw_{m,s}((2m-1,1))\chi_{(m,1)}(t)\,,
\]
and define the test function
\[
w(t):=
\begin{cases}
\hat w(t) & t\in(2m-1,m]\cr
\hat w(t)+(s-\hat w(1)) & t\in(m,1)
\end{cases}\,.
\]
Clearly, $w\in BV((2m-1,1))$, with 
\[
Dw=w_{m,s}'\mathcal L^1\res(2m-1,1)
+(D^sw_{m,s}((2m-1,1))+s-\hat w(1))\delta_{\{t=m\}},
\]
$w(2m-1)=0$, and $w(1)=s$. Since $\hat w(1)=Dw_{m,s}((2m-1,1))=w_{m,s}(1)-w_{m,s}(2m-1)$
we have that $sc^{-}G_{m,s}(w)<sc^{-}G_{m,s}(w_{m,s})$ if either the boundary data are not attained
by $w_{m,s}$ or $D^sw_{m,s}$ is not concentrated on $t=m$.
Indeed, computing the energy for $w$ shows that the contributions for the energy
of $w_{m,s}$ corresponding 
either to the boundary terms or to the singular part of the distributional derivative, 
which are both multiplied by $\hatk_m^{\sfrac12}(t)$ for some $t\neq m$, have been 
either moved to or concentrated in $t=m$, which is the unique minimum point of the 
same function $\hatk_m^{\sfrac12}(t)$.
\medskip

\noindent{\bf Step 3. Euler-Lagrange equation and structural properties. } 
By Step~2 any minimizer $w_{m,s}$ is $SBV((2m-1,1))$, 
$w_{m,s}(2m-1)=0$, $w_{m,s}(1)=s$, and $D^sw_{m,s}$ is at most concentrated on $t=m$.
Note also that by truncation necessarily $w_{m,s}\in[0,s]$ $\calL^1$-a.e. on $(2m-1,1)$.

Assume first that $w_{m,s}\in W^{1,1}((2m-1,1))$, then by using smooth outer variations,
i.e. $w_{m,s}+\eps \phi$ with $\phi\in C^1_c((2m-1,1))$, we obtain 
the following Euler-Lagrange equation
\begin{equation}\label{e:lambda ms}
\frac{\hatk_m(t)w'_{m,s}(t)}{\big(\hatk_m(t)|w_{m,s}'(t)|^2+\PotDann_m(1-t)\big)^{\sfrac12}}=\lambda_{m,s} 
\end{equation}
for $\calL^1$-a.e. $t\in (2m-1,1)$ and for some $\lambda_{m,s}\in \R$. Actually, $\lambda_{m,s}\in(0,\infty)$  as $w_{m,s}\in[0,s]$
$\calL^1$-a.e. on $(2m-1,1)$ and attains the boundary data, 
in turn implying that $w_{m,s}'>0$ $\calL^1$-a.e. on $(2m-1,1)$, and thus 
$\lambda_{m,s}\leq \hatk^{\sfrac12}(m)$ (cf. \eqref{e:lambda ms}). 
Furthermore, \eqref{e:lambda ms} shows that $w_{m,s}$ is strictly increasing and odd symmetric with respect to $m$. Solving $w'_{m,s}$ in 
\eqref{e:lambda ms} and integrating on $(2m-1,1)$ gives
\begin{equation*}
    s=2\int_m^1\Big(\frac{\lambda_{m,s}^2\PotDann(1-t)}{\hatk(t)(\hatk(t)-\lambda^2_{m,s})}\Big)^{\sfrac12}\dt.
\end{equation*}
Therefore, we conclude the uniqueness of $\lambda_{m,s}$ that satisfies \eqref{e:el}, and 
in turn that $w_{m,s}$ is the only possible minimizer of $sc^{-}G_{m,s}$ among $W^{1,1}$ functions, using the strict monotonicity of 
\begin{equation}\label{e:strictly monotonicity}
     \lambda \mapsto \int_m^1\Big(\frac{\lambda^2\PotDann(1-t)}{\hatk(t)(\hatk(t)-\lambda^2)}\Big)^{\sfrac12}dt\,.
\end{equation}

Instead, if $w_{m,s}\in SBV\setminus W^{1,1}((2m-1,1))$, % with $[w_{m,s}](m)>0$. 
we consider again smooth outer variations separately on $(2m-1,m)$ and $(m,1)$ to infer the Euler-Lagrange equation in \eqref{e:lambda ms} with constants $\lambda^-_{m,s}$ and $\lambda^+_{m,s}$ on each interval, respectively. 
Since, $w_{m,s}$ takes value in $[0,s]$ $\calL^1$-a.e. on $(2m-1,1)$, necessarily $\lambda_{m,s}^\pm\geq 0$.

Next, we use variations on the whole interval that move the value of the jump point in 
$t=m$, i.e.~$w_{m,s}+\eps \phi$ with $\phi\in C^1((2m-1,m)\cup(m,1))\cap SBV((2m-1,1))$, and $\phi(2m-1)=\phi(1)=0$. Note that for $\eps$ sufficiently small $[w_{m,s}+\eps\phi](m)$ has the same sign of $[w_{m,s}](m)$, therefore we get %using $[w_{m,s}](m)>0$, 
\begin{equation*}
 \lambda^-_{m,s}\phi(m^-)-\lambda^+_{m,s}\phi(m^+)\pm\hatk^{\sfrac12}(m)(\phi(m^+)-\phi(m^-))=0,
\end{equation*}
where the sign $+$ corresponds to the case $[w_{m,s}](m)>0$ and $-$ to the other case. 
In particular, it follows that $\lambda^+_{m,s}=\lambda^-_{m,s}=\pm\hatk^{\sfrac12}(m)$,
and since they are both non negative we conclude \eqref{e:el}.
Hence, $w_{m,s}$ is odd symmetric with respect to $t=m$.
Moreover, $w_{m,s}(m^-)<\frac s2$ since otherwise having assumed $w_{m,s}\in SBV\setminus W^{1,1}((2m-1,1))$ then $w_{m,s}(m^-)>\frac s2$ and the function $\hat w=w_{m,s}\chi_{(2m-1,m)}+(w_{m,s}\vee w_{m,s}(m^-))\chi_{[m,1)}$ would be $W^{1,1}((2m-1,1))$ with
$sc^{-}G_{m,s}(\hat w)<sc^{-}G_{m,s}(w_{m,s})$.
Therefore, $w_{m,s}$ is strictly increasing and it is the only possible minimizer of $sc^{-}G_{m,s}$ 
belonging to $SBV\setminus W^{1,1}((2m-1,1))$ in view of \eqref{e:el}.
\end{proof}
By taking into account the odd symmetry of the minimizer $w_{m,s}$ with respect to $t=m$ 
we rewrite the minimal value of the energy $sc^{-}G_{m,s}$ in an alternative and more efficient form. With this aim it is convenient to introduce the function
$\mathscr{G}:\{(m,\lambda)\in (0,1)^2:m\in(0,1),\,\lambda\in(0,\hatk^{\sfrac12}(m)]\}\to[0,\infty]$ defined by
\begin{align}\label{e:mathscrGs}
    \mathscr{G}(m,\lambda)&:=2\int_m^1\Big(\frac{\hatk(t)-\lambda^2}{\hatk(t)}\PotDann(1-t)\Big)^{\sfrac12}\dt
    +2\lambda\int_m^1 \Big(\frac{\lambda^2 \PotDann(1-t) }{\hatk(t)(\hatk(t)-\lambda^2)}\Big)^{\sfrac12}\dt\notag\\
    &=:A(m,\lambda)+\lambda B(m,\lambda)\,. 
\end{align}
Note that $\mathscr{G}$ is infinite if and only if $B$ is infinite, and the only points in which this can happen are of the type
$(m,\hatk^{\sfrac12}(m))$. In addition we set $\mathscr{G}(0,0):=2\Psi(1)$.
\begin{corollary}\label{c:regolarità minimo old}
Assume (Hp~$1$)-(Hp~$5$). For every $m\in(0,1)$, and $s\in(0,\infty)$,
if $w_{m,s}$ denotes the minimizer of $sc^{-}G_{m,s}$ in 
Theorem~\ref{t:regolarità minimo old}, we have
\begin{align}\label{e:mathfrakGs}
\mathfrak{G}_s(m)&:=sc^{-}G_{m,s}(w_{m,s})\notag\\
&=2\int_{m}^1\Big(\hatk(t)|w_{m,s}'|^2+\PotDann(1-t)\Big)^{\sfrac12}\dt
 + 2\hatk^{\sfrac12}(m)(w_{m,s}(m^+)-\frac s2)\,, 
\end{align}
and either 
\begin{equation}\label{e:s-el1}
w_{m,s}\in W^{1,1}((2m-1,1))\Longleftrightarrow 
s=B(m,\lambda_{m,s})\,,
\end{equation}
where $\lambda_{m,s}\in(0,\hatk^{\sfrac12}(m)]$ is as in Theorem~\ref{t:regolarità minimo old},
and 
\begin{equation}\label{e:Gs W11}
%sc^{-}G_{m,s}(w_{m,s})
\mathfrak{G}_s(m)=2\int_m^1\Big(\frac{\hatk(t)\PotDann(1-t)}{\hatk(t)-\lambda^2_{m,s}}\Big)^{\sfrac12}\dt
=\mathscr{G}(m,\lambda_{m,s})=A(m,\lambda_{m,s})+\lambda_{m,s}s\,,
\end{equation}
or
\begin{equation}\label{e:s-el2}
w_{m,s}\in SBV\setminus W^{1,1}((2m-1,1))\Longleftrightarrow
s>B(m,\hatk^{\sfrac12}(m))\,,
\end{equation}
and 
\[
w_{m,s}(m^+)=s-\frac12 B(m,\hatk^{\sfrac12}(m))\,,
\]
and 
\begin{align}\label{e:Gs SBV}
&\mathfrak{G}_s(m)=2\int_m^1\Big(\frac{\hatk(t)\PotDann(1-t)}{\hatk(t)-\hatk(m)}\Big)^{\sfrac12}\dt
+\hatk^{\sfrac12}(m)(s-B(m,\hatk^{\sfrac12}(m)))\notag\\
&=\mathscr{G}(m,\hatk^{\sfrac12}(m))+\hatk^{\sfrac12}(m)(s-B(m,\hatk^{\sfrac12}(m)))
=A(m,\hatk^{\sfrac12}(m))+\hatk^{\sfrac12}(m)s\,.
\end{align}
\end{corollary}
\begin{remark}
We stress that the value $\mathscr{G}(m,\hatk^{\sfrac12}(m))$ in \eqref{e:Gs SBV} is necessarily finite due to \eqref{e:s-el2}, analogously due to \eqref{e:Gs W11} if $\lambda_{m,s}=\hatk^{\sfrac12}(m)$.
\end{remark}
\begin{proof}[Proof of Corollary~\ref{c:regolarità minimo old}]
The proof of \eqref{e:mathfrakGs} follows from a simple calculation using the explicit form of $sc^{-}G_{m,s}$ in \eqref{e:rilassato old}, the odd symmetry of $w_{m,s}$
and the fact that it is increasing. Moreover, the same properties of $w_{m,s}$ yield that
\[
B(m,\lambda_{m,s})=2\int_m^1\Big(\frac{\lambda_{m,s}^2\PotDann(1-t)}{\hatk(t)(\hatk(t)-\lambda^2_{m,s})}\Big)^{\sfrac12}\dt
=\int_{2m-1}^1w'_{m,s}\dt\,,
\]
implying \eqref{e:s-el1} and \eqref{e:s-el2}. 

Finally, formulas \eqref{e:Gs W11} and \eqref{e:Gs SBV} follows by direct computations
using \eqref{e:el} and \eqref{e:mathfrakGs}.
\end{proof}

\subsection{A key auxiliary function}

Under suitable assumptions on $s$, $\hatk$ and $\PotDann$ we will show that there is a 
unique value $m_s$ such that $g(s)=sc^-G_{m_s,s}(w_{m_s,s})$. With this aim we define
the auxiliary function $\Phi:(0,1)\to [0,\infty]$ as follows
\begin{equation}\label{e:la psi piccola}
    \Phi(x):= B(x,\hatk^{\sfrac12}(x))\,.
\end{equation}
The importance of $\Phi$ is related to the fact that \eqref{e:s-el1} rewrites equivalently 
as $\Phi(m)\geq s$, and \eqref{e:s-el2} as $\Phi(m)<s$.
In the next proposition we determine several properties of $\Phi$, as well as that its
definition is well posed provided that the following assumption on $\hatk$ which strengthens (Hp~$5$) holds:
\begin{itemize}
 \item[(Hp~$5'$)] $\hatk \in C^1((0,1))$ with $\hatk'(t)>0$ for every $t\in(0,1)$.
\end{itemize}
Recall that $\Psi$ is the function introduced in \eqref{e:Psi}.
\begin{proposition}\label{p:BCI}
Assume (Hp~$1$)-(Hp~$4$), and (Hp~$5'$). Then
\begin{itemize}
\item[(i)] $\Phi\in C^0((0,1))$; 
\item[(ii)] if the following limit exists 
\begin{equation}\label{e:sqrthatf derivata}
\lim_{x\to0^+}(\hatk^{\sfrac12})'(x)=:(\hatk^{\sfrac12})'(0^+) \in[0,\infty]\,,
\end{equation}
then $\Phi$ has right limit in $x=0$ given by
\begin{equation}\label{e:Phi0+}
\lim_{x\to0^+}\Phi(x)=\frac{\pi \PotDann^{\sfrac12}(1)}{2(\hatk^{\sfrac12})'(0^+)}\,;
\end{equation}
\item[(iii)] %setting $\eta:=\Psi\circ\hatk^{-1}$, 
if $\hatk(1^-)=\FailureS^2\in(0,\infty)$, and there is $p>2$ such that 
for some $z\in (0,\FailureS^2)$
\begin{equation}\label{e:F}
(0,\FailureS^2)\ni t\mapsto F(t):= \frac{(\Psi\circ\hatk^{-1})'(\FailureS^2-t)}{(\FailureS^2-t)^{\sfrac12}}
\in L^p((0,z))\,,
\end{equation}
then 
\begin{equation}\label{e:limite Phi =0}
\lim_{x\to 1^-}\Phi(x)=0\,.
\end{equation}
If, in addition, \eqref{e:F} holds for every $z\in(0,1)$ then $\Phi\in C^{\sfrac12-\sfrac1p}_{\textup{loc}}([0,1))$;
\item[(iv)] if $[0,1]\ni x\mapsto (x(1-x))^{\sfrac12}F(xt)$ is strictly decreasing (non-increasing, 
strictly increasing, non-decreasing) for $\calL^1$-a.e. $t\in(0,1)$, then $\Phi$ is strictly increasing (non-decreasing, strictly decreasing, non-increasing);
\item[(v)] if $(\hatk \circ \Psi^{-1})^{\sfrac12}$ is convex, 
%where $\xi:=\hatk \circ \Psi^{-1}$, 
then $\Phi$ is strictly decreasing.
\end{itemize}  
\end{proposition}
\begin{proof} 
Let $\gamma\in(0,\sfrac12)$, then there is $\delta\in(0,\sfrac\gamma2)$ such that 
$\hatk(t)-\hatk(y)\geq \frac{\hatk'(y)}{2}(t-y)$ for every $y\in[\gamma,1-\gamma]$ and $t$ such that $|t-y|\leq\delta$. Therefore, recalling that $\hatk'(t)>0$ for every $t\in(0,1)$, 
$\Phi$ is well defined for every $x\in(0,1)$. 

In addition, if $x_j\to x$ we have for every $\varepsilon>0$ sufficiently small
\[
\Phi(x_j)=2\int_{x_j}^{x_j+\varepsilon}\Big(\frac{\hatk(x_j)\PotDann(1-t)}{\hatk(t)(\hatk(t)-\hatk(x_j))}\Big)^{\sfrac12}\dt
+2\int_{x_j+\varepsilon}^1\Big(\frac{\hatk(x_j)\PotDann(1-t)}{\hatk(t)(\hatk(t)-\hatk(x_j))}\Big)^{\sfrac12}\dt=:I_j+K_j\,.
\]
It is clear from the continuity assumption on $\hatk$ that 
\[
\lim_jK_j=2\int_{x+\varepsilon}^1\Big(\frac{\hatk(x)\PotDann(1-t)}{\hatk(t)(\hatk(t)-\hatk(x))}\Big)^{\sfrac12}\dt\,.
\]
Moreover, we estimate $I_j$ as follows
\[
0\leq I_j\leq C\sup_{(x-\varepsilon,x+2\varepsilon)} (\hatk\hatk')^{-\sfrac12}
\int_{x_j}^{x_j+\varepsilon}(t-x_j)^{-\sfrac12}\dt=
 2C\sup_{(x-\varepsilon,x+2\varepsilon)} (\hatk\hatk')^{-\sfrac12}\varepsilon^{\sfrac12}\,,
\]
and the conclusion follows.

Next, we prove that the right limit of $\Phi$ in $0$ exists and the equality in \eqref{e:Phi0+} holds.
With this aim, note that for every $\delta>0$ being $\hatk$ and $\PotDann$ continuous on $[0,1]$ and $\hatk(0)=0$,  we have
\[
\Phi(x)=2(\PotDann^{\sfrac12}(1)+o(1))
\int_{x}^\delta\Big(\frac{\hatk(x)}{\hatk(t)(\hatk(t)-\hatk(x))}\Big)^{\sfrac12}\dt+o(1)
\]
as $x\to 0^+$. Changing variable with $\tau=(\sfrac{\hatk(t)}{\hatk(x)})^{\sfrac12}$
and applying the Mean Value Theorem gives some $\bar{t}\in[x,\delta]$ such that
\begin{align*}
\Phi(x)
&=4(\PotDann^{\sfrac12}(1)+o(1))\frac{\hatk^{\sfrac12}(\bar{t})}{\hatk'(\bar{t})}
\int_{1}^{(\frac{\hatk(\delta)}{\hatk(x)})^{\sfrac12}}
\frac{1}{\tau(\tau^2-1)^{\sfrac12}}\dd\tau+o(1)\\
&=\frac{2(\PotDann^{\sfrac12}(1)+o(1))}{(\hatk^{\sfrac12})'(\bar{t})}
\left(\arctan\left(\Big(\frac{\hatk(\delta)}{\hatk(x)}\Big)^{\sfrac12}+\Big(\frac{\hatk(\delta)}{\hatk(x)}-1\Big)^{\sfrac12}\right)
-\frac\pi4\right)+o(1)\\
\end{align*}
as $x\to 0^+$, where the last equality follows from an elementary integration argument, and the fact that $\hatk^{\sfrac12}\in C^1((0,1))$ (cf. the assumption on $\hatk$).
Eventually \eqref{e:Phi0+} follows at once by letting first $x\to 0^+$ and then $\delta\to0^+$,
using \eqref{e:sqrthatf derivata}.

Next, we prove \eqref{e:limite Phi =0}.
%With fixed $x\in (0,1)$, by means of the one-dimensional area formula 
%\cite[Theorem~3.4.6]{AT04} 
Recalling that $\hatk(1^-)=\FailureS^2\in(0,\infty)$, changing variables with $t=\hatk^{-1}(\FailureS^2-\tau)$ 
in the definition of $\Phi$ (cf. \eqref{e:mathscrGs} and \eqref{e:la psi piccola}) 
and letting $\lambda=\FailureS^2-\hatk(x)$, we deduce the following Abel integral equation 
involving the function $F$ defined in the statement
\begin{equation}\label{e:Phi Abel}
%    \frac{\Phi(x)}{2\hatk^{\sfrac12}(x)}
%    = \int_0^{\FailureS^2-\hatk(x)}\frac{(\Psi\circ\hatk^{-1})'(\FailureS^2-\tau)}{(\FailureS^2-\tau)^{\sfrac12}(\FailureS^2-\hatk(x)-\tau)^{\sfrac12}}\dd\tau
\frac{\Phi(\hatk^{-1}(\FailureS^2-\lambda))}{2(\FailureS^2-\lambda)^{\sfrac12}}
= \int_0^{\lambda}\frac{(\Psi\circ\hatk^{-1})'(\FailureS^2-\tau)}{(\FailureS^2-\tau)^{\sfrac12}
(\FailureS^2-\hatk(x)-\tau)^{\sfrac12}}\dd\tau
=\int_0^\lambda\frac{F(\tau)}{(\lambda-\tau)^{\sfrac12}}\dd\tau\,,
\end{equation}
%In particular, we deduce the following Abel integral equation 
%involving the function $F$ defined in the statement
%\begin{equation}\label{e:Phi Abel}
%    \frac{\Phi(\hatk^{-1}(\FailureS^2-\lambda))}{2(\FailureS^2-\lambda)^{\sfrac12}}
%    =\int_0^\lambda\frac{F(\tau)}{(\lambda-\tau)^{\sfrac12}}\dd\tau\,, 
%\end{equation}
(see \cite{Vessella} for references on such a topic). By assumption $F\in L^p((0,z))$, $p>2$, 
for some $z\in (0,\FailureS^2)$, so that if $\frac1q+\frac1p=1$, $q\in(1,2)$, and if $\lambda\in(0,z)$
we get
\[
|\Phi(\hatk^{-1}(\FailureS^2-\lambda))|\leq2(\FailureS^2-\lambda)^{\sfrac12}
\|F\|_{L^p((0,z))}\Big(\frac2{2-q}\lambda^{\frac{2-q}{2}}\Big)^{\sfrac1q}\,,
\]
and thus the limit of $\Phi(\hatk^{-1}(\FailureS^2-\lambda))$ as $\lambda\to 0^+$ is infinitesimal, 
corresponding to the left limit of $\Phi$ in $1$.
In addition, if $F\in L^p((0,z))$, $p>2$, for every $z\in (0,\FailureS^2)$ the left hand side 
of \eqref{e:Phi Abel} is $C^{\sfrac12-\sfrac1p}_{\textup{loc}}([0,\FailureS^2))$ 
by \cite[Theorem~4.1.4]{Vessella}.

Next, we prove the monotonicity assertions. We start with the claim in item (iv), 
focusing on the strictly increasing case, the other cases being analogous. 
With this aim we use the change of variable $\tau=t z$ 
in \eqref{e:Phi Abel} to get 
\begin{equation*}
    \Phi(\hatk^{-1}(1-z))
    =2(z(1-z))^{\sfrac12}\int_0^1 \Big(\frac{t}{1-t}\Big)^{\sfrac12}F(tz)\dt, 
\end{equation*}
and the conclusion follows being $[0,1]\ni z\mapsto (z(1-z))^{\sfrac12}F(tz)$ strictly decreasing 
for $\calL^1$-a.e. $t\in(0,1)$, and $\hatk$ strictly increasing.

To establish the claim in item (v) we first use %the one-dimensional area formula \cite[Theorem~3.4.6]{AT04} with 
the change of variable $\tau=\frac{\hatk(t)}{\hatk(x)}$ 
in the definition of $\Phi$ (cf. \eqref{e:la psi piccola}) to deduce that 
\begin{align} \label{e:Phi change variable}%{e:Phi change variable0}
 \Phi(x)
=2\int_1^{\frac{1}{\hatk(x)}} \Big(\frac{\hatk(x)}{\tau(\tau-1)}\Big)^{\sfrac12}
 \frac{\Psi'(\hatk^{-1}(\hatk(x)\tau))}{\hatk'(\hatk^{-1}(\hatk(x)\tau))}\dd\tau
 %\notag\\
=\int_1^{\frac{1}{\hatk(x)}} 
\frac{\zeta'\big((\hatk(x)\tau)^{\sfrac12}\big)}{\tau(\tau-1)^{\sfrac12}}\dd\tau,
\end{align}
where $\zeta$ denotes the inverse of $(\hatk\circ\Psi^{-1})^{\sfrac12}$, 
namely $\zeta(t)=\Psi(\hatk^{-1}(t^2))$. We conclude using the fact that $\hatk$ is 
strictly increasing, and $(\hatk\circ\Psi^{-1})^{\sfrac12}$ is convex and increasing. 
\end{proof}
\begin{remark}
More generally, using the one-dimensional area formula (see \cite[Theorem~3.4.6]{AT04}),
the definition of $\Phi$ is well posed and the conclusions in item (iii)
hold assuming $\hatk^{-1}\in W^{1,1}((0,1))$.
\end{remark}

We note that formula \eqref{e:Phi0+} corresponds to \cite[ $($3.11$)$ in Proposition 3.7]{BI23}

in case $\PotDann(t)=t^2$ (there it is equivalently stated in terms of the differentiability in $t=0$ of the function $f$ in \eqref{e:f}). 

The monotonicity of $\Phi$ is a key tool in the analysis below of the variational problem defining $g$ as it implies that, for every $s\geq0$, the energies $\mathfrak{G}_s$ are minimized either in the interior of $[0,1]$ in the strictly decreasing case, or at the boundary in the increasing case. In this respect, the convexity assumption on $(\hatk\circ\Psi^{-1})^{\sfrac12}$ is equivalent to the assumption used in \cite{BCI,BI23} (with $Q(t)=\PotDann(t)=t^2$) which in our notation reads as the convexity of $(\hatk\circ(2(\Psi(1)-\Psi))^{-1})^{\sfrac12}$. We remark that the concavity of $(\hatk\circ(2(\Psi(1)-\Psi))(t^{\sfrac12})$ has been used in \cite{Lammen2025} (with $Q(t)=\PotDann(t)=t^2$) to infer the concavity of the corresponding energy.
We have not been able to show that such a condition implies neither that $\Phi$ is increasing nor that the energy $\mathfrak{G}_s$ is concave.
Despite this, using item~(iv) simple calculations show that the explicit example given in \cite{Lammen2025} (namely, $\hatk(t)=1-(1-t)^4$ and $\PotDann(t)=t^2$) yields that $\Phi$ is increasing.

\begin{remark}\label{r:strictly decreasing assumption}
The fact that $\Phi$ is strictly decreasing is also satisfied, for instance, 
in the following particular cases for every non-decreasing damage potential $\PotDann$ 
\begin{itemize}
    \item[(i)] $\hatk(t):=t^{\alpha}$, with $\alpha \geq 2$;
    \item[(ii)] $\hatk(t):=\frac{e^t-1}{e-1}$; % and $\PotDann$ non-decreasing;
\end{itemize}
where $\Psi$ is defined in \eqref{e:Psi}. For instance, it suffices to take $Q=\PotDann$ 
and $\hatf(t)=\hatk(t)=t^\alpha$ in case (i), analogously in case (ii). 
To prove the claim we use the change of variable $\frac{t}{x}=\tau$, for every $x\in (0,1)$, to get
\begin{equation*}
    \Phi(x)=2\int_1^{\frac{1}{x}} x\Big(\frac{\hatk(x)\PotDann(1-\tau x)}{\hatk(\tau x)(\hatk(\tau x)-\hatk(x))}\Big)^{\sfrac12}\dd \tau\,.
\end{equation*}
In case $\hatk(t)=t^{\alpha}$, we obtain for every $x\in (0,1)$
\begin{equation*}
   \Phi(x)=2\int_1^{\frac{1}{x}} x^{1-\frac{\alpha}{2}}\Big(\frac{\PotDann(1-\tau x)}{\tau^{\alpha}(\tau^{\alpha}-1)}\Big)^{\sfrac12}\dd \tau\,,
\end{equation*}
which is strictly decreasing, being $\alpha\geq 2$ and $\PotDann$ non-decreasing by assumption. 
Similarly, one can argue for the exponential.
Note also that the result in item~(iii) of Proposition~\ref{p:BCI} applies to $\hatk(t)=t^\alpha$ for $\alpha\geq 2$.
\end{remark}

\subsection{Dependence of $g$ on the parameters of the phase-field model}

In this section we show the explicit dependence of $g$ on the damage potential and the degradation
function. With this aim, we prove a technical result which will be instrumental in what follows.
\begin{lemma}\label{l:derivate energia}
Assume (Hp~$1$)-(Hp~$5$). 
Let $s\in(0,\infty)$ and $0<m_0<m_1<1$, consider $\lambda_{m,s}\in(0,\hatk^{\sfrac12}(m)]$ in Theorem~\ref{t:regolarità minimo old},
and $\mathfrak{G}_s$ in \eqref{e:mathfrakGs}. Then
\begin{enumerate}
\item if $\lambda_{m,s}\in(0,\hatk^{\sfrac12}(m))$ for every $m\in(m_0,m_1)$, then $\mathfrak{G}_s\in C^1((m_0,m_1))$ and is strictly decreasing.
Moreover, $(m_0,m_1)\mapsto\lambda_{m,s}\in C^1((m_0,m_1))$ and is strictly increasing;
\item if $\lambda_{m,s}=\hatk^{\sfrac12}(m)$
for every $m\in(m_0,m_1)$ and $\hatk\in C^1((m_0,m_1))$ with $\hatk'(t)>0$ for every $t\in (0,1)$, then $\mathfrak{G}_s\in C^1((m_0,m_1))$ and increasing.

More precisely, let $w_{m,s}$ be the function defined in Theorem~\ref{t:regolarità minimo old} then
\begin{enumerate}
\item if $w_{m,s}\in SBV\setminus W^{1,1}((2m-1,1))$ for every $m\in(m_0,m_1)$ then 
$\mathfrak{G}_s$ is strictly increasing on $(m_0,m_1)$.
\item if $w_{m,s}\in W^{1,1}((2m-1,1))$ for some $m\in(m_0,m_1)$ then $\mathfrak{G}_s'(m)=0$. 
\end{enumerate}
\end{enumerate}
\end{lemma}
\begin{proof} We start with the first case noting that in particular \eqref{e:s-el1} necessarily holds, and thus $w_{m,s}\in W^{1,1}((2m-1,1))$ 
in view of Corollary~\ref{c:regolarità minimo old}. Moreover, \eqref{e:Gs W11} yields that
\begin{equation}\label{e:mathfrakGs W11}
\mathfrak{G}_s(m)=2\int_{m}^1 \Big(\frac{\hatk(t)\PotDann(1-t)}{\hatk(t)-\lambda_{m,s}^2}\Big)^{\sfrac12}\dt\,.
\end{equation}
Next, we claim that $(m_0,m_1)\mapsto \lambda_{m,s}$ is continuously differentiable and strictly increasing
thanks to the implicit function theorem applied to the function 
$B$ defined in \eqref{e:mathscrGs} 
defined on the open set 
\[
T:=\left\{(m,\lambda) \; : \; m\in (m_0,m_1), \; \lambda \in (0,\hatk^{\sfrac12}(m))\right\}.
\]
Indeed, item (i) in Theorem~\ref{t:regolarità minimo old} ensures that for every $m\in (m_0,m_1)$ 
the unique value of $\lambda$ %$\overline{\lambda}_m$ 
such that $B(m,\lambda)=s$ is $\lambda_{m,s}\in(0,\hatk^{\sfrac12}(m))$, and
\begin{align}\label{e:derivata lambda B}
\frac{\partial B}{\partial\lambda}(m,\lambda)&
=2\int_m^1\Big(\frac{\PotDann(1-t)}{\hatk(t)(\hatk(t)-\lambda^2)}\Big)^{\sfrac12}\dt
+2\lambda^2\int_m^1\Big(\frac{\PotDann(1-t)}{\hatk(t)(\hatk(t)-\lambda^2)^3}\Big)^{\sfrac12}\dt
\notag\\
&=2\int_m^1\Big(\frac{\hatk(t)\PotDann(1-t)}{(\hatk(t)-\lambda^2)^3}\Big)^{\sfrac12}\dt>0\,,
\end{align}
Thus, by the implicit function theorem
\begin{align*}
%\frac{d\lambda_{m,s}}{dm}
\lambda'_{m,s}&=-\frac{\partial B}{\partial m}(m,\lambda_{m,s})
\left(\frac{\partial B}{\partial\lambda}(m,\lambda_{m,s})\right)^{-1}\\
&=2\left(\frac{\lambda^2_{m,s}\PotDann(1-m)}{\hatk(m)(\hatk(m)-\lambda^2_{m,s})}\right)^{\sfrac12}\left(\frac{\partial B}{\partial\lambda}(m,\lambda_{m,s})\right)^{-1}>0\,.
\end{align*}
Therefore, $(m_0,m_1)\mapsto \lambda_{m,s}$ is strictly increasing.
In turn, this implies that $\mathfrak{G}_s$ in \eqref{e:mathfrakGs W11} is differentiable on $(m_0,m_1)$.
Indeed, using \eqref{e:Gs W11}, namely $\mathfrak{G}_s(m)=A(m,\lambda_{m,s})+\lambda_{m,s}s$,
\eqref{e:s-el1} and \eqref{e:derivata lambda B} we get
\begin{align}\label{e:derivative G SBV}
\mathfrak{G}_s'(m)&=\left(\frac{\partial A}{\partial\lambda}(m,\lambda_{m,s})+s\right)\lambda_{m,s}'+
\frac{\partial A}{\partial m}(m,\lambda_{m,s})=\frac{\partial A}{\partial m}(m,\lambda_{m,s})<0
\end{align}
as from the very definition of $A$ in \eqref{e:mathscrGs} easy calculations lead to
\[
\frac{\partial A}{\partial\lambda}(m,\lambda)=-B(m,\lambda)\,,\qquad
\frac{\partial A}{\partial m}(m,\lambda)=-2\left(\frac{\hatk(m)-\lambda^2}{\hatk(m)}\PotDann(1-m)\right)^{\sfrac12}\,.
\]

Next, we deal with the second case. In particular note that $w_{m,s}$ can be either $W^{1,1}$ or $SBV\setminus W^{1,1}$,
and in both cases $\lambda_{m,s}=\hatk^{\sfrac12}(m)$. We claim that $\mathfrak{G}_s\in C^1((m_0,m_1))$ with
\begin{equation}\label{e:derivative if G}
  \mathfrak{G}_s'(m)=\frac{\hatk'(m)}{2\hatk^{\sfrac12}(m)}(s-\Phi(m))
\end{equation}
for every $m\in (m_0,m_1)$. Given this for granted, the right hand side in \eqref{e:derivative if G} 
is either null if $w_{m,s}\in W^{1,1}$ thanks to $s=B(m,\hatk^{\sfrac12}(m))=\Phi(m)$ (cf. \eqref{e:s-el1})
and \eqref{e:la psi piccola}), or is strictly positive if $w_{m,s}\in SBV\setminus W^{1,1}$
 $s>B(m,\hatk^{\sfrac12}(m))=\Phi(m)$ (cf. \eqref{e:s-el2}),
and thus we conclude the monotonicity of $\mathfrak{G}_s$.

To prove \eqref{e:derivative if G} define the map $L:T'\to \R$ by
\begin{equation*}
 L(m,\rho)= 2\int_{m}^1\Big(\frac{(\hatk(t)-\hatk(\rho))\PotDann(1-t)}{\hatk(t)}\Big)^{\sfrac12}\dt+s\hatk^{\sfrac12}(m),
 \end{equation*}
 where $T':=\{(m,\rho):\,  m\in (m_0,m_1),\; \rho\in (0,m)\}$. Note that $L\in C^0(\overline{T'})$, 
 and $\mathfrak{G}_s(m)=L(m,m)$ on $(m_0,m_1)$.
 Thanks to the assumption $\hatk\in C^1((m_0,m_1))$, then $L\in C^1(T')$ with
 \begin{equation}\label{e:partial L partial m}
     \frac{\partial L}{\partial m}(m,\rho)=
     -2\left(\frac{(\hatk(m)-\hatk(\rho))\PotDann(1-m)}{\hatk(m)}\right)^{\sfrac12}
     +\frac{s\hatk'(m)}{2\hatk^{\sfrac12}(m)}\,,
 \end{equation}
 and
 \begin{equation}\label{e:partial L partial rho}
   \frac{\partial L}{\partial \rho}(m,\rho)=
   -\hatk'(\rho)\int_m^1\Big(\frac{\PotDann(1-t)}{\hatk(t)(\hatk(t)-\hatk(\rho))}\Big)^{\sfrac12} \dt
   =-\hatk'(\rho)\frac{B(m,\hatk^{\sfrac12}(\rho))}{2\hatk^{\sfrac12}(\rho)}.
 \end{equation}
 Being $\hatk^{-1}$ locally Lipschitz on $(m_0,m_1)$, we can extend $\nabla L$ to 
 $\overline{T}$ continuously. Therefore, an elementary argument shows that 
 $\mathfrak{G}_s\in C^1((m_0,m_1))$ and
\begin{equation*}
\mathfrak{G}_s'(m)=\frac{\partial L}{\partial m}(m,m)+\frac{\partial L}{\partial \rho}(m,m)
=\frac{\hatk'(m)}{2\hatk^{\sfrac12}(m)}(s-\Phi(m))\,,
\end{equation*}
where we have used \eqref{e:partial L partial m}, \eqref{e:partial L partial rho}, and \eqref{e:la psi piccola} 
to deduce \eqref{e:derivative if G}.
\end{proof}
We are now ready to infer some interesting consequences of the previous statement for the minimum problem defining $g$.
With this aim, we introduce
\begin{equation}\label{e:def sfrac}
    \sfratt:=\sup\{s \in [0,\infty) \; | \; g(s)<2\Psi(1) \}\in(0,\infty],
\end{equation}
where $\Psi$ is defined in \eqref{e:Psi}. The value $\sfratt$ is well defined in 
view of either item (ii) in Proposition~\ref{p:lepropdig} in case $\hatk(1^-)<\infty$ 
or item (iii) in Proposition~\ref{p:lepropdig infty} otherwise. Moreover,
$g(s)<2\Psi(1)$ for every $s\in[0,\sfratt)$
by taking into account that $g$ is non-decreasing.
\begin{proposition}\label{l:Gs reduction}
Assume (Hp~$1$)-(Hp~$4$), and (Hp~$5'$). Then
 \begin{itemize}
    \item[(i)] $g(s)=\FailureS s$ for every $s\in[0,\inf\Phi]$. In particular, $\inf\Phi=0$ if
    $\hatk(1^-)=\infty$;
    \item[(ii)] $g(s)=({\displaystyle{\inf_{\Phi^{-1}(s)}\mathfrak{G}_s}})\wedge(\displaystyle{\liminf_{m\to 0^+}}\mathfrak{G}_s)\wedge
    (\displaystyle{\liminf_{m\to 1^-}}\mathfrak{G}_s)$ for every $s\in(\inf\Phi,\sup\Phi)$;
    \item[(iii)] $g(s)=2\Psi(1)$ for every $s\in[\sup\Phi,\infty)$.
 \end{itemize}
 In particular, $\sfratt\in[\inf\Phi,\sup\Phi]$.  
\end{proposition}
\begin{proof} \noindent{{\bf Step 1.}}
Assume $\inf\Phi>0$ otherwise the statement is trivial. Let 
$s\in[0,\inf\Phi)$, then $w_{m,s}\in W^{1,1}((2m-1,1))$ for every $m\in(0,1)$. 
Therefore, the energy $\mathfrak{G}_s$ is strictly decreasing on $(0,1)$ by Lemma~\ref{l:derivate energia}, so that we may use formula \eqref{e:mathfrakGs} and the monotonicity of $\hatk$ to bound the energy from below as follows:
\begin{equation*}%\label{e:mathfrakGs geq s lb}
\mathfrak{G}_s(m)\geq2\hatk^{\sfrac12}(m)\int_m^1w_{m,s}'\dt=\hatk^{\sfrac12}(m)s\,.
\end{equation*}
%By recalling that $\hatk(1^-)=\FailureS^2\in(0,\infty]$, the inequality 
%in \eqref{e:mathfrakGs geq s lb} yields that 
Therefore, $\hatk(1^-)<\infty$ and  
\begin{equation}\label{e:mathfrakGs geq s}
    \lim_{m\to1^-}\mathfrak{G}_s(m)\geq \FailureS s\,.
\end{equation}
Hence, we deduce that $g(s)\geq \FailureS s$ for every $s\in(\sfratt,\inf\Phi)$. 
On the other hand, 
$g(s)\leq \FailureS s\wedge(2\Psi(1))$ by item (ii) in Proposition~\ref{p:lepropdig}. 
Hence, $\sfratt\geq\inf\Phi$ and $g(s)=\FailureS s$ for $s\in(\sfratt,\inf\Phi)$. By continuity of $g$ we have that
the last equality extends to $s=\inf\Phi$. 

%To conclude we prove item (vi) arguing by contradiction thus supposing that
%$\inf\Phi>0$. If $s\in (0,\inf \Phi)$ we have that $w_{m,s}\in W^{1,1}((2m-1,1))$ for every $m\in (0,1)$. Therefore, the energy $\mathfrak{G}_s$ is strictly decreasing on $(0,1)$ by Lemma~\ref{l:derivate energia},  
%so that we may use formula \eqref{e:mathfrakGs} and the monotonicity of $\hatk$ to bound the energy from below as follows:
%\[
%\mathfrak{G}_s(m)\geq2\hatk^{\sfrac12}(m)\int_m^1w_{m,s}'\dt=\hatk^{\sfrac12}(m)s\,.
%\]
%By recalling that $\hatk(1^-)=\infty$ we obtain a contradiction because $g$ is bounded.

\smallskip

\noindent{{\bf Step 2.}}
Let $s\in(\inf\Phi,\sup\Phi)$, then consider the open sub-level set $\{m\in(0,1):\,\Phi(m)<s\}$ 
and the open super-level set $\{m\in(0,1):\,\Phi(m)>s\}$. 
Let $(m_0,m_1)\subset\{m\in(0,1):\,\Phi(m)<s\}$ be a connected component, %so that $\Phi(a)=\Phi(b)=s$.
then $w_{m,s}\in SBV\setminus W^{1,1}((2m-1,1))$ for every $m\in (m_0,m_1)$.
By case (ii)-(a) in Lemma~\ref{l:derivate energia} the energy is increasing on $(m_0,m_1)$ so that 
\[
\inf_{(m_0,m_1)}\mathfrak{G}_s=\lim_{m\to m_0^+}\mathfrak{G}_s(m).
\]
In particular, Lemma~\ref{l:derivate energia} itself implies the continuity of $\mathfrak{G}_s$ on $(m_0,m_1)$,
and if $m_0>0$  we claim that $\mathfrak{G}_s$ is right continuous in $m_0$. Indeed, using \eqref{e:Gs SBV} we obtain
\begin{equation}\label{e:ciruzzo}
\lim_{m\to m_0^+}\mathfrak{G}_s(m)=\lim_{m\to m_0^+}A(m,\hatk^{\sfrac12}(m))+\hatk^{\sfrac12}(m_0)s=\mathfrak{G}_s(m_0)\,,
\end{equation}
where in the last equality we have used that $A\in C^0(\{(m,\lambda)\in (0,1)^2:m\in(0,1),\,\lambda\in(0,\hatk^{\sfrac12}(m)]\})$. Therefore, 
it follows that $\mathfrak{G}_s$ is continuous on $[m_0,1)$. Note that in this case $\Phi(m_0)=s$.

Now let $(m_0,m_1)\subset\{m\in(0,1):\,\Phi(m)>s\}$ be a connected component, then $w_{m,s}\in W^{1,1}((2m-1,1))$ for every $m\in (m_0,m_1)$.
By case (i) in Lemma~\ref{l:derivate energia} the energy is strictly decreasing on $(m_0,m_1)$ so that 
\[
\inf_{(m_0,m_1)}\mathfrak{G}_s=\lim_{m\to m_1^-}\mathfrak{G}_s(m)\,,
\]
with $\mathfrak{G}_s$ continuous on $(m_0,m_1)$. We claim that if $m_1<1$ then $\mathfrak{G}_s$ is left-continuous in $m_1$ as well.
Indeed, we note that $(m_0,m_1)\mapsto \lambda_{m,s}$ is continuously differentiable and strictly increasing (cf. Lemma~\ref{l:derivate energia}).
Then necessarily $\displaystyle{\lim_{m\to m_1^-}}\lambda_{m,s}=\hatk^{\sfrac12}(m_1)$ by \eqref{e:s-el1} as $\Phi(m_1)=s$ by continuity of $\Phi$ in 
$(0,1)$. To conclude we argue as follows:
\begin{align*}
    \mathfrak{G}_s(m)&=2\int_{m}^1 \Big(\frac{\hatk(t)\PotDann(1-t)}{\hatk(t)-\lambda_{m,s}^2}\Big)^{\sfrac12}\dt\\
    &=2\int_{m}^{m_1} \Big(\frac{\hatk(t)\PotDann(1-t)}{\hatk(t)-\lambda_{m,s}^2}\Big)^{\sfrac12}\dt
+2\int_{m_1}^1 \Big(\frac{\hatk(t)\PotDann(1-t)}{\hatk(t)-\lambda_{m,s}^2}\Big)^{\sfrac12}\dt\,.
\end{align*}
The first summand above is infinitesimal by the local Lipschitz continuity of $\hatk^{-1}$.
Indeed, if $K^2$ is the maximum value of $\frac{\hatk(t)\PotDann(1-t)}{(\hatk^{-1})'(t)}$ 
on $[\frac12\hatk^{\sfrac12}(m_1),\hatk^{\sfrac12}(m_1)]$ then for $m$ sufficiently close to $m_1$
we get
\[
\int_{m}^{m_1} \Big(\frac{\hatk(t)\PotDann(1-t)}{\hatk(t)-\lambda_{m,s}^2}\Big)^{\sfrac12}\dt\leq
K\int_{m}^{m_1} \frac1{\big(t-\hatk^{-1}(\lambda_{m,s}^2)\big)^{\sfrac12}}\dt
\leq 2K(m_1-\hatk^{-1}(\lambda_{m,s}^2))^{\sfrac12}\,,
\]
and the conclusion follows by the continuity as $\displaystyle{\lim_{m\to m_1^-}}\lambda_{m,s}=\hatk^{\sfrac12}(m_1)$.
Therefore, we have
\begin{align*}
\lim_{m\to m_1^-}\mathfrak{G}_s(m)&
=\lim_{m\to m_1^-}2\int_{m_1}^1 \Big(\frac{\hatk(t)\PotDann(1-t)}{\hatk(t)-\lambda_{m,s}^2}\Big)^{\sfrac12}\dt
\\&=2\int_{m_1}^1 \Big(\frac{\hatk(t)\PotDann(1-t)}{\hatk(t)-\hatk(m_1)}\Big)^{\sfrac12}\dt=\mathfrak{G}_s(m_1),
\end{align*}
by the monotone convergence theorem.

Therefore, we have shown that $g(s)\geq(\inf_{\Phi^{-1}(s)}\mathfrak{G}_s)\wedge(\displaystyle{\liminf_{m\to 0^+}}\mathfrak{G}_s)\wedge
(\displaystyle{\liminf_{m\to 1^-}}\mathfrak{G}_s)$. The opposite inequality is trivial as $g(s)=\inf_{(0,1)}\mathfrak{G}_s$.
\smallskip

\noindent{{\bf Step 3.}}
Next, consider any $s>\sup\Phi$, then by \eqref{e:s-el2} necessarily $w_{m,s}\in SBV\setminus W^{1,1}((2m-1,1))$ for every $m\in(0,1)$.
Therefore, $\mathfrak{G}_s$ is strictly increasing on $(0,1)$ by case 2 in Lemma~\ref{l:derivate energia}, and thus
\[
g(s)=\inf_{(0,1)}\mathfrak{G}_s=\lim_{m\to0^+}\mathfrak{G}_s(m)=\lim_{m\to 0^+}A(m,\hatk^{\sfrac12}(m))
\]
where in the last inequality we have used the explicit formula \eqref{e:Gs SBV} for $\mathfrak{G}_s$ and the fact that
the second summand in there is infinitesimal.
On the other hand, by using Fatou's lemma we conclude the continuity of 
$[0,1)\ni m\mapsto A(m,\hatk^{\sfrac12}(m))$ in $m=0$ and the thesis. Indeed, by Fatou's lemma we have
\begin{align}\label{e:limite Gs 0}
g(s)&=\lim_{m\to 0^+}A(m,\hatk^{\sfrac12}(m))\notag\\&
=2\lim_{m\to0^+}\int_0^1\Big(\frac{(\hatk(t)-\hatk(m))\vee 0}{\hatk(t)}\,\PotDann(1-t)\Big)^{\sfrac12}\dt\geq2\Psi(1)\,.
\end{align}
The conclusion for $s>\sup\Phi$ follows from the estimate $g(s)\leq \FailureS s\wedge(2\Psi(1))$ if $\hatk(1^-)<\infty$ (cf. item (ii) in Proposition~\ref{p:lepropdig}), and the estimate $g(s)\leq g_0(s)\leq 2\Psi(1)$
if $\hatk(1^-)=\infty$ (cf. item (ii) in Proposition~\ref{p:lepropdig infty}). 
By continuity of $g$ we obtain the equality $g(\sup\Phi)=2\Psi(1)$. Therefore, $\sfratt\leq\sup\Phi$.
\end{proof}

Next, we assume $\Phi$ to be strictly decreasing, as discussed for instance in Remark~\ref{r:strictly decreasing assumption}, to identify a unique minimizing value (defined in terms of $\Phi$ itself) for the problem defining $g(s)$ introduced in \eqref{e:doppio inf rilassato G}.
\begin{theorem}\label{t:teo la g esplicita}
Assume (Hp~$1$)-(Hp~$4$), and (Hp~$5'$),
and that $\Phi:(0,1)\to(0,\infty)$ in \eqref{e:la psi piccola} is strictly decreasing. Then, the following hold
\begin{itemize}
    \item[(i)] $\sfratt=\sup\Phi=\Phi(0^+)$;
    
    \item[(ii)] $g(s)=\FailureS s$ for every $s\in[0,\inf\Phi]$. Moreover, for every $s\in (\inf\Phi,\sfratt)$ 
    let $m_s:=\Phi^{-1}(s)$, then $w_{m_s,s}\in W^{1,1}((2m_s-1,1))$, $\lambda_{m_s,s}=\hatk^{\sfrac12}(m_s)$ and
    \begin{equation}\label{e:la g esplicita}
    g(s)=\mathfrak{G}_s(m_s)\,,
\end{equation}
and $m_s$ is the unique value $m$ such that $g(s)=\mathfrak{G}_s(m)$;

\item[(iii)] $g\in C^1([0,\infty))$ if $\hatk(1^-)<\infty$, and 
$g\in C^1((0,\infty))\cap C^0([0,\infty))$ if $\hatk(1^-)=\infty$, with   
\begin{equation}\label{e:la g' esplicita}
    g'(s)=\hatk^{\sfrac12}(m_s)\,,
\end{equation}
for $s\in[\inf\Phi,\sfratt]$, $g'(s)=\FailureS$ for $s\in[0,\inf\Phi]$, and $g'(s)=0$ for $s\geq\sfratt$.
In particular, $g$ is concave (actually strictly concave on $[\inf\Phi,\sfratt]$).
Moreover, $\FailureS\sfratt\geq2\Psi(1)$ if $\hatk(1^-)<\infty$.
\end{itemize}
\end{theorem}
\begin{proof} Being $\Phi$ strictly decreasing $\inf\Phi=\Phi(1^-)$ and $\sup\Phi=\Phi(0^+)$,
then Proposition~\ref{l:Gs reduction} yields that $\sfratt\in[\Phi(1^-),\Phi(0^+)]$, and
$\Phi((0,1))=(\Phi(1^-),\Phi(0^+))$ by continuity of $\Phi$ (cf. Proposition~\ref{p:BCI}).

{\bf Step 1: minimality of the relaxed energy at $m=m_s$.}
Fix $s\in (\Phi(1^-),\Phi(0^+))$, and let $m_s:=\Phi^{-1}(s)\in(0,1)$. 
Then $w_s:=w_{m_s,s}\in W^{1,1}((2m_s-1,1))$
by \eqref{e:s-el1}, with $\lambda_{m_s,s}=\hatk^{\sfrac12}(m_s)$,
$w_{s}(2m_s-1)=0$, $w_{s}(1)=s$, and moreover in view of \eqref{e:el}
\begin{equation*}
    w_s(t)=\frac{s}{2}+\int_{m_s}^t \Big(\frac{\hatk(m_s)\PotDann(1-\tau)}{\hatk(\tau)(\hatk(\tau)-\hatk(m_s))}\Big)^{\sfrac12}\dd\tau\,.
\end{equation*}
In addition, the following claims hold due to the strict monotonicity of $\Phi$, using that
$B(m,\cdot)$ is strictly increasing for all $m$, and $B(\cdot,\lambda)$ is strictly decreasing for all $\lambda$
\begin{itemize}
    \item[(i)] $w_{m,s}\in W^{1,1}((2m-1,1))$ and $\lambda_{m,s}\in(0,\hatk^{\sfrac12}(m))$ for every $m\in (0,m_s)$,
    \item[(ii)] $w_{m,s}\in SBV\setminus W^{1,1}((2m-1,1))$ for every $m\in (m_s,1)$.
\end{itemize}
To conclude \eqref{e:la g esplicita} we apply case (ii) in Proposition~\ref{l:Gs reduction} noting that
$\mathfrak{G}_s(m_s)<(\displaystyle{\liminf_{m\to 0^+}}\mathfrak{G}_s)\wedge
    (\displaystyle{\liminf_{m\to 1^-}}\mathfrak{G}_s)$ (the latter inferior limits are actually limits being $\Phi$ monotone)
    thanks to Lemma~\ref{l:derivate energia} and items (i) and (ii) above. 

The equality $g(s)=\FailureS s$ for $s\in[0,\Phi(1^-)]$ follows directly from case (i) in Proposition~\ref{l:Gs reduction}.
\medskip

\noindent{\bf Step 2. Differentiability of $g$.}
To establish the differentiability of $g$ and prove \eqref{e:la g' esplicita}, fix $s_1,s_2 \in (\Phi(1^-),\Phi(0^+))$. 
Thanks to \eqref{e:la g esplicita}, using the competitor $w=\frac{s_1}{s_2}w_{s_2}\in 
W^{1,1}((2m_2-1,1))$, where $m_2:=m_{s_2}$ and $w_{s_2}=w_{s_2,m_2}$, we obtain that
\begin{align}\label{e:stima gs1s2}
    g(s_1)&=\inf_{m\in(0,1)}\inf_{L^1((2m-1,1))}sc^{-}G_{m,s_1}\leq sc^{-}G_{m_2,s_1}(w)\notag\\
    &=2\int_{m_2}^1\Big(\hatk(t)|w'|^2+\PotDann(1-t)\Big)^{\sfrac12}\dt= g(s_2)\notag\\
    &+2\left(\frac{s_1^2}{s_2^2}-1\right)\int_{m_{2}}^1\frac{\hatk(t)|w_{s_2}'|^2}{\big((\frac{s_1}{s_2})^2\hatk(t)|w_{s_2}'|^2+\PotDann(1-t)\big)^{\sfrac12}+\big(\hatk(t)|w_{s_2}'|^2+\PotDann(1-t)\big)^{\sfrac12}}\dt.
\end{align}
In particular, if $s_1<s_2$ we get
\begin{align}\label{e: rapporto incrementale 1}
 & \frac{g(s_1)-g(s_2)}{s_1-s_2}\notag \\ & \geq 2\frac{s_2+s_1}{s_2^2} \int_{m_{2}}^1\frac{\hatk(t)|w_{s_2}'|^2}{\big((\frac{s_1}{s_2})^2\hatk(t)|w_{s_2}'|^2+\PotDann(1-t)\big)^{\sfrac12}+\big(\hatk(t)|w_{s_2}'|^2+\PotDann(1-t)\big)^{\sfrac12}} \dt,
\end{align}
therefore using \eqref{e:el} (see also \eqref{e:lambda ms}), we infer
\begin{align*}
    %g'_-(s_2)
    \liminf_{s_1\to s_2^-} \frac{g(s_1)-g(s_2)}{s_1-s_2}&\geq \frac{2}{s_2}\int_{m_2}^1 \frac{\hatk(t)|w_{s_2}'|^2}{\big(\hatk(t)|w_{s_2}'|^2+ \PotDann(1-t)\big)^{\sfrac12}}\dt\\
    &=\frac{2}{s_2}\int_{m_2}^1\hatk^{\sfrac12}(m_2)w_2' \dt=\hatk^{\sfrac12}(m_2).
\end{align*}
Using \eqref{e:stima gs1s2} with $s_1>s_2$ we obtain 
\begin{align}\label{e: rapporto incrementale 2}
 &\limsup_{s_2\to s_1^-}\frac{g(s_2)-g(s_1)}{s_2-s_1} 
 \leq\limsup_{s_2\to s_1^-} \frac{s_1+s_2}{s^2_2}\int_{m_2}^1 \frac{\hatk(t)|w_{s_2}'|^2}{\big(\hatk(t)|w_{s_2}'|^2+ \PotDann(1-t)\big)^{\sfrac12}}\dt\notag \\
&=\limsup_{s_2\to s_1^-}\frac{s_2+s_1}{s^2_2}\int_{m_2}^1\hatk^{\sfrac12}(m_2)w_2' \dt=\limsup_{s_2\to s_1^-}
\frac{s_2+s_1}{2s_2}\hatk^{\sfrac12}(m_2)=\hatk^{\sfrac12}(m_1),
\end{align}
where in the last equality we have used the continuity of $\Phi^{-1}$.
In particular, we have shown that at each point $s\in(\Phi(1^-),\Phi(0^+))$ the left derivative 
of $g$ exists and equals $\hatk^{\sfrac12}(m_s)$. 

The conclusion for the right derivative can be deduced analogously by passing to the inferior limit as
$s_2\to s_1^+$ in \eqref{e: rapporto incrementale 1}, and to the superior limit as $s_1\to s_2^-$ in
\eqref{e: rapporto incrementale 2}.

To conclude, we need only to prove the differentiability of $g$ in $\Phi(1^-)$ if
$\hatk(1^-)<\infty$, and in $\Phi(0^+)$ in general assuming the latter finite. Indeed,
$g'(s)=\FailureS$ for $s\in[0,\Phi(1^-))$ if $\hatk(1^-)<\infty$,  
and $g'(s)=0$ if $s>\Phi(0^+)$. 
Thus, the formula for $g'$ in \eqref{e:la g' esplicita} 
and the continuity of $\hatk$ respectively in $\Phi(1^-)$ if $\hatk(1^-)<\infty$
and in $\Phi(0^+)$ in general, imply that
\[
\lim_{s\to\Phi(1^-)^+}g'(s)=\hatk^{\sfrac12}(1)=\FailureS,\qquad\lim_{s\to\Phi(0^+)^-}g'(s)=\hatk^{\sfrac12}(0)=0\,.
\]
Furthermore, we deduce that $\sfratt=\Phi(0^+)$ thanks to \eqref{e:la g' esplicita} and 
as by definition $g(s)=2\Psi(1)$ for every $s\geq\sfratt$.

Moreover, $(\Phi(1^-),\Phi(0^+))\ni s\mapsto m_s$ is non-increasing as $\Phi$ is so,
in turn implying the concavity of $g$ by \eqref{e:la g' esplicita}. Eventually, 
if $\hatk(1^-)<\infty$, necessarily $\FailureS\sfratt\geq 2\Psi(1)$ as $g'(0)=\FailureS$ and $g$ is concave.
\end{proof}

In some cases under a more restrictive set of assumptions an expression for $g$ similar to that
in \eqref{e:la g esplicita}, and the differentiability of $g$, have been first established in 
\cite[Proposition~3.3]{BI23} and \cite[Proposition~3.6]{BI23}, respectively.

\begin{remark}
 We point out that for $s\in(\inf\Phi,\sup\Phi)$ the energy $\mathfrak{G}_s$ in Theorem~\ref{t:teo la g esplicita}
 turns out to be $C^1((0,1))$ with $\mathfrak{G}_s'(m_s)=0$ thanks to item (ii) case (b) in Lemma~\ref{l:derivate energia}.
 %together with the continuity of $\mathfrak{G}_s$ established in Step~2 above (cf. \eqref{e:derivative G SBV} and \eqref{e:derivative if G}).
 \end{remark}

\begin{remark}
If $\Phi$ is non-increasing on $(0,1)$, one cannot expect $g$ to be of class $C^1$. Indeed,  
within the hypothesis of Proposition~\ref{l:Gs reduction}, assume that $\Phi(m)= s_0$ if 
$m\in[m_0,m_1]$, for some $0<m_0<m_1<1$, and strictly decreasing otherwise.
Then, from the condition $\Phi(m)= s_0$ for every $m\in[m_0,m_1]$ it follows that the corresponding
minimizers $w_{m,s}\in W^{1,1}$. Defining $\overline{m}:=\inf\{m\in[m_0,m_1]:\,\lambda_{m,s}=\hatk^{\sfrac12}(m)\}$, and setting $\overline{m}=m_1$ if the latter set is empty, we have that 
$g(s_0)=\mathfrak{G}_{s_0}(m)$ for every $m\in [\overline{m},m_1]$. Indeed, 
item (ii) case (b) in Lemma~\ref{l:derivate energia} yields that $\mathfrak{G}_s'(m)=0$
on the interval $(\overline{m},m_1)$, if $\overline{m}<m_1$.

Otherwise, on $(0,1)\setminus [m_0,m_1]$ we may argue as in Theorem~\ref{t:teo la g esplicita}
to obtain formula \eqref{e:la g esplicita}.
Moreover \eqref{e:la g' esplicita} hold for every $s\in (0,s_0)\cup (s_0,1)$, where $\sfratt=\Phi(0^+)$. Therefore, $g\in C^1((0,s_0)\cup(s_0,1))$, and $g$ is not differentiable in $s_0$ as
\begin{equation*}
    \lim_{s\to s_0^-} %\frac{g(s)-g(s_0)}{s-s_0}
    g'(s)=\hatk^{\sfrac12}(m_1) \neq 
    \lim_{s\to s_0^+} %\frac{g(s)-g(s_0)}{s-s_0}
    g'(s)=\hatk^{\sfrac12}(m_0)\,.
\end{equation*}
\end{remark}

Next, we deal with the case $\Phi$ increasing and $\hatk(1^-)<\infty$,
we show that the cohesive law $g$ always corresponds to that in the Dugdale cohesive model.
\begin{theorem}\label{t:teo la g esplicita Dugdale}
Assume (Hp~$1$)-(Hp~$4$), and (Hp~$5'$), $\hatk(1^-)<\infty$, and that
$\Phi:(0,1)\to(0,\infty)$ in \eqref{e:la psi piccola} is non-decreasing. 
 Then, $g(s)=\FailureS s\wedge (2\Psi(1))$ for every $s\geq 0$.
\end{theorem}
\begin{proof} The proof is similar to that of Theorem~\ref{t:teo la g esplicita}.  
Recall that by continuity $\Phi((0,1))=(\Phi(0^+),\Phi(1^-))$.
Next, we use Proposition~\ref{l:Gs reduction} to infer that
$g(s)=\FailureS s$ for all $s\in[0,\Phi(0^+)]$ by case (i) there, and
 $g(s)=2\Psi(1)$ for all $s\geq\Phi(1^-)$ by case (iii) there.
Thus, we are left with considering the case $\Phi(0^+)<\Phi(1^-)$ and
$s\in (\Phi(0^+),\Phi(1^-))$.
Let $m\in[m_0,m_1]=\Phi^{-1}(s)\in(0,1)$ and note that $w_{m,s}\in W^{1,1}((2m-1,1))$
by \eqref{e:s-el1}, with $\lambda_{m,s}=\hatk^{\sfrac12}(m)$,
and moreover in view of \eqref{e:el}
\begin{equation*}
    w_{m,s}(t)=\frac{s}{2}+\int_{m}^t \Big(\frac{\hatk(m)\PotDann(1-\tau)}{\hatk(\tau)(\hatk(\tau)-\hatk(m))}\Big)^{\sfrac12}\dd\tau\,.
\end{equation*}
In addition, the following claims hold due to the monotonicity of $\Phi$ 
\begin{itemize}
    \item[(i)] $w_{m,s}\in SBV\setminus W^{1,1}((2m-1,1))$ for every $m\in (0,m_0)$
    \item[(ii)] $w_{m,s}\in W^{1,1}((2m-1,1))$ and $\lambda_{m,s}\in(0,\hatk^{\sfrac12}(m))$ for every $m\in (m_1,1)$.
\end{itemize}
To conclude $g(s)=\FailureS s\wedge(2\Psi(1))$ we apply case (ii) in Proposition~\ref{l:Gs reduction} noting that
$\mathfrak{G}_s(m_s)>(\displaystyle{\liminf_{m\to 0^+}}\mathfrak{G}_s)\wedge
(\displaystyle{\liminf_{m\to 1^-}}\mathfrak{G}_s)$ (the latter inferior limits are actually limits by
monotonicity) thanks to Lemma~\ref{l:derivate energia} and items (i) and (ii) above. Therefore, we have that
\[
g(s)=\inf_{(0,1)}\mathfrak{G}_s(m)=
(\lim_{m\to 0^+}\mathfrak{G}_s(m))\wedge(\lim_{m\to1^-}\mathfrak{G}_s(m))\,.
\]
The first limit is easily seen to be bigger than $2\Psi(1)$ arguing as to deduce
\eqref{e:limite Gs 0} in Proposition~\ref{l:Gs reduction}). In addition,
the second limit above is bigger than $\FailureS s$ arguing as to deduce \eqref{e:mathfrakGs geq s} in 
Proposition~\ref{l:Gs reduction}). Therefore,  $g(s)\geq \FailureS s\wedge(2\Psi(1))$ for every $s\in(\Phi(0^+),\Phi(1^-))$,
the reverse inequality is contained in item (ii) in Proposition~\ref{p:lepropdig}.
\end{proof}

\section{Reconstructing the cohesive law \texorpdfstring{$g$}{...}}%{Reconstructing the phase-field model %for a given $g$
\label{s:sezione 5.3}

In this section we show how to assign a given cohesive law $g_0$, either with a linear or a superlinear behaviour for small amplitudes of the jump, by appropriately choosing the parameters of the phase-field model in \eqref{functeps}.
More precisely, either fixing the damage potential and choosing appropriately the degradation function or vice versa, we consider the corresponding functionals $\Functeps$ and show that $g_0$ equals the surface energy density of their $\Gamma$-limit according to Theorem~\ref{t:finale}.

A scaling argument shows that the assumption $\varphi'(0^+)=1$ in 
Theorems~\ref{t:hatf fixed}, \ref{t:potdann fixed}, \ref{t:potdann fixed infinito}, 
and \ref{t:hatf fixed infinito} below can be easily dispensed with. The functions in the conclusions 
of those statements would then be depending on the value $\varphi'(0^+)\neq1$.

\subsection{The linear case}

We start with reconstructing cohesive laws that behave linearly for small jump amplitudes.
More precisely, let 
\begin{itemize}
\item[(Hp~$6$)] $g_0\in C^1([0,\infty))$, $g_0^{-1}(0)=\{0\}$, $g_0$ bounded and non-decreasing;
\item[(Hp~$7$)] $g_0$ concave,  $g_0'(0)=\FailureS\in(0,\infty)$, $g'_0$ is strictly decreasing on $[0,(\sfratt)_0)$ where 
\[
(\sfratt)_0:=\sup \{s \in [0,\infty):\, g_0(s)<\lim_{s\to\infty}g_0(s)\}\in(0,\infty]\,.
\]
\end{itemize}
In addition, it is convenient to define the auxiliary function $R:[0,1]\to[0,g_0(\infty)]$ by
\begin{equation}\label{e:La R grande}
    R(t):=
    \begin{cases}
        g_0\big((g_0')^{-1}((\FailureS^2-t)^{\sfrac12})\big) \; & \textup{if $t\in [0,\FailureS^2)$} \\
        g_0((\sfratt)_0) \; & \textup{if $t=\FailureS^2$}.
     \end{cases}
\end{equation}
As a consequence of (Hp~$6$), (Hp~$7$) and the definition of $(\sfratt)_0$, 
$R$ is strictly increasing. Moreover, $g_0'(s)\to0$ for $s\to (\sfratt)_0^-$ as $g_0\in C^1([0,\infty))$ is concave and strictly decreasing, so that
$g_0(\infty)<\infty$. Therefore, $R$ is continuous. We further assume that
\begin{itemize}
\item[(Hp~$8$)] $R$ is convex on $[0,\FailureS^2]$ and $W^{1,p}((0,\FailureS^2))$ for some $p>2$.
\end{itemize}
Having introduced $R$, for every $\tau\in [0,\FailureS^2]$ consider 
\begin{equation}\label{e:la phi della R}
    \phi(\tau):=\frac1{\pi}\int^\tau_0\frac{R'(t)}{(\tau-t)^{\sfrac12}} \dt\,.
\end{equation} 
By Abel's inversion theorem \cite[Theorem~1.A.1]{Vessella} (see also formula (1.B.1ii)), 
 $\phi$ is the unique $L^1((0,1))$ function satisfying the following Abel integral equation 
\begin{equation*}
    R(t)=\int^{t}_0 \frac{\phi(\tau)}{(t-\tau)^{\sfrac12}}\dd\tau\,,
\end{equation*}
for every $t\in [0,\FailureS^2]$. This is a consequence of the smoothness of $R$ and of $R(0)=0$.
 Moreover, the regularity, the strict monotonicity and the convexity of $R$ in (Hp~$8$) 
 yield that $\phi \in C^0([0,\FailureS^2])$, $\phi^{-1}(0)=\{0\}$, and $\phi$ is strictly increasing
 (cf. \cite[Theorem~4.1.4]{Vessella}). More generally,  by changing variables
 in the very definition of $\phi$ above (i.e., $t=\tau y$), the strict monotonicity of $\phi$
 follows provided $R'(t)t^{\sfrac12}$ is non-decreasing and not constant on $(0,\FailureS^2)$.

We are now ready to reconstruct $g_0$. 
In the ensuing two results we will always choose $\PotDann=\FailureS^2 Q$, so that $\hatk=\FailureS^2\hatf$.
We start fixing $\hatk$ with $\hatk'(t)>0$ for every $t\in(0,1)$, and find the appropriate $\PotDann$.
In particular, $\hatk(t)=\FailureS^2t^2$ satisfies all the conditions below.
\begin{theorem}\label{t:hatf fixed}
Assume that $g_0$ satisfy (Hp~$6$)-(Hp~$7$), $R$ satisfy (Hp~$8$), and $\hatk_0$ is such that 
$\hatk_0(0)=0$, $\hatk_0(1^-)=\FailureS^2$, $\hatk_0^{\sfrac12}\in C^1([0,1])$ with $\hatk_0'(t)>0$ 
for every $t\in(0,1)$, and concave in a left neighborhood of $t=1$ (in particular, $\hatk_0$ satisfies
(Hp~$5'$)).

Then there exists $\PotDann_0\in C^0([0,1])$ such that $\hatf(t):=\FailureS^{-2}\hatk_0(t)$, $\FailureS^2Q(t)=\PotDann(t):=\PotDann_0(t)$ on $[0,1]$ satisfy (Hp~$1$)-(Hp~$2$), 
and for any $\varphi$ satisfying (Hp~$3$)-(Hp~$4$) with $\varphi'(0^+)=1$ the corresponding 
functionals $\Functeps$ in \eqref{functeps} $\Gamma$-converge to $F_{\FailureS}$ in \eqref{F0} with $g=g_0$.
\end{theorem}
\begin{proof}
Let us first assume $\FailureS=1$. Let $\phi$ be defined in \eqref{e:la phi della R},
and then set for every $t\in [0,1]$
\begin{equation}\label{e:PotDann zero}
    \PotDann_0(1-t):=\big((\hatk_0^{\sfrac12}(t))'\phi(1-\hatk_0(t))\big)^2\,.
\end{equation}
Clearly, $\PotDann_0\in C^0([0,1])$ and with the choices $\hatf:=\hatk_0$ and $Q=\PotDann:=\PotDann_0$, 
(Hp~$1$) and (Hp~$2$) are satisfied as $\phi$ is strictly increasing and $\hatk_0^{\sfrac12}$ is concave in a left neighborhood of $t=1$.
Moreover, \eqref{e:FailureS def} holds with $\FailureS=1$. 
In particular, choosing any $\varphi$ that satisfies (Hp~$3$) and (Hp~$4$) with 
$\varphi'(0^+)=1$, by Theorem~\ref{t:finale} $(\Functeps)_\eps$ $\Gamma$-converges to $F_1$, where the surface energy density $g$ is given by 
\begin{equation*}
      g(s):=\inf_{(u,v)\in \calulu_s} \int_0^1 \Big(\hatk_0(v)|u'|^2+\PotDann_0(1-v)|v'|^2\Big)^{\sfrac12}\dx.
\end{equation*}
Note that $\Phi(1^-)=0$ by \eqref{e:limite Phi =0} in Proposition~\ref{p:BCI}. Indeed, 
using \eqref{e:PotDann zero} we have
\[
F(t)=\frac{(\Psi_0\circ\hatk_0^{-1})'(1-t)}{(1-t)^{\sfrac12}}
=\frac{\phi(t)}{2(1-t)}\in C^0([0,z])
\]
for every $z\in(0,1)$  being $\phi\in C^0([0,1])$. In addition, $\Phi$ is strictly decreasing as
\eqref{e:PotDann zero} implies that for every $t\in[0,1]$
\begin{equation*}%\label{e:Phi vs phi}
\big(\Psi_0(\hatk_0^{-1}(t^2))\big)'=2t\big(\Psi_0\circ \hatk_0^{-1}\big)'(t^2)=\phi(1-t^2)\,,
\end{equation*}
the claim thus follows from item (v) in Proposition \ref{p:BCI}, and the fact that $\phi$ 
is strictly increasing.
Therefore, we may apply Theorem~\ref{t:teo la g esplicita} to deduce for every $s\in (0,\sfratt)$
\begin{align}\label{e:Abel equation}
    g(s)&=2\int_{\hatk_0^{-1}((g'(s))^2)}^1 \Big(\frac{\hatk_0(t)\PotDann_0(1-t)}{\hatk_0(t)-(g'(s))^2}\Big)^{\sfrac12}\dt\notag\\
    &=2\int_{0}^{1-(g'(s))^2} \Big((1-r)\frac{\PotDann_0(1-\hatk_0^{-1}(1-r))}{1-(g'(s))^2-r}\Big)^{\sfrac12}\frac1{\hatk_0'(\hatk_0^{-1}(1-r))}\dd r\,,
\end{align}
having used the change of variables $r=1-\hatk_0(t)$. Therefore, on setting $\lambda=1-(g'(s))^2$ 
and using the very definition of $\PotDann_0$ we conclude for every $\lambda\in [0,1)$
\begin{align*}
g((g')^{-1})\big((1-\lambda)^{\sfrac12}\big)
=\int_0^\lambda \frac{\phi(r)}{(\lambda-r)^{\sfrac12}} \dd r= R(\lambda)
=g_0\big((g_0')^{-1}((1-\lambda)^{\sfrac12}\big)\,,
\end{align*}
where we have used that $g'$ is invertible in view of \eqref{e:la g' esplicita}.
Inverting $g\circ(g')^{-1}$ on $[0,g(\sfratt))$ and $g_0\circ(g_0')^{-1}$ on $[0,g_0((\sfratt)_0))$ we conclude that 
$(g^{-1})'=(g_0^{-1})'$ on $[0,g_0((\sfratt)_0)\wedge g(\sfratt))$.
From this, using the very definitions of $(\sfratt)_0$ and $\sfratt$,
together with $g_0'((\sfratt)_0)=g'(\sfratt)=0$, $g_0'>0$ on $(0,(\sfratt)_0)$, 
and $g'>0$ on $(0,\sfratt)$, we deduce that $(\sfratt)_0=\sfratt$. Being $g^{-1}(0)=g_0^{-1}(0)=0$, we conclude that $g(s)=g_0(s)$ for every $s\in [0,\infty)$.

If $\FailureS\neq 1$, apply the previous argument to $\widetilde{g}_0:=\FailureS^{-1}g_0$ and $\widetilde{\hatf}_0(t)=t^2$, to find $\widetilde{\PotDann}_0$ such that $\widetilde{g}_0$ is the surface energy density of the $\Gamma$-limit of the corresponding functionals $\Functeps$ with $\widetilde{Q}=\widetilde{\PotDann}:=\widetilde{\PotDann}_0$ and $\widetilde{\hatf}:=\widetilde{\hatf}_0$. The conclusion for $g_0$
then follows by taking $\hatf(t):=t^2$, $\PotDann:=
\FailureS^2\widetilde{\PotDann}_0$, and $Q:=\widetilde{\PotDann}_0$.
\end{proof}

Next, given $g_0$, we fix $\PotDann$ (and $Q=\FailureS^{-2} \PotDann$), and determine $\hatk$,  
equivalently $\hatf$.
\begin{theorem}\label{t:potdann fixed}
Let $g_0$ satisfy (Hp~$6$)-(Hp~$7$), $R$ satisfy (Hp~$8$), and $\PotDann_0\in C^0([0,1],[0,\infty))$ , 
$\PotDann_0^{-1}(\{0\})=\{0\}$, $\PotDann_0$ increasing in a right neighbourhood of $0$ and such that $2\Psi_0(1)=g_0(\infty)$, where %\Psi_0(x):=\int_0^x \PotDann_0^{\sfrac12}(1-t)\dt$. 
$\Psi_0$ is defined as in \eqref{e:Psi} integrating $\PotDann_0$.

Then, there exists $\hatf_0\in C^0([0,1])\cap C^1((0,1))$ with $\hatf_0'>0$ on $(0,1)$, such that
$\hatf:=\hatf_0$, $\FailureS^2 Q=\PotDann:=\PotDann_0$ satisfy (Hp~$1$), (Hp~$2$), $\hatk:=\FailureS^2\hatf$ satisfies (Hp~$5'$), and for any $\varphi$ satisfying (Hp~$3$)-(Hp~$4$) with $\varphi'(0^+)=1$, the functionals $\Functeps$ in \eqref{functeps} $\Gamma$-converge to $F_{\FailureS}$ in \eqref{F0} with $g=g_0$.
\end{theorem}
\begin{proof}
Let us first assume $\FailureS=1$. With $\phi$ defined in \eqref{e:la phi della R}, 
let $\hatf_0:[0,1]\to[0,1]$ be for every $t\in [0,1]$
\begin{equation*}
    \hatf^{-1}_0(1-t):=\Psi_0^{-1}\left(\frac{g_0(\infty)}{2}-\frac{1}{2}\int_0^t \frac{\phi(\tau)}{(1-\tau)^{\sfrac12}}\dd \tau\right).
\end{equation*}
Then $\hatf_0\in C^0([0,1])$, $\hatf_0$ is strictly increasing, $\hatf_0(0)=0$ as $R(1)=g_0(\infty)$, $\hatf_0(1)=1$ as $2\Psi_0(1)=g_{0}(\infty)$, 
$\hatf_0\in C^1((0,1))$ with $\hatf_0'>0$ in $(0,1)$, $\hatf_0^{-1}\in W^{1,1}((0,1))$ and for every $t\in (0,1)$
\begin{equation}\label{e:derivative hatf0-1}
  2\big((1-t)\PotDann_0(1-\hatf_0^{-1}(1-t))\big)^{\sfrac12}(\hatf_0^{-1})'(1-t)=\phi(t).
\end{equation}
We observe that setting $\hatf=\hatk:=\hatf_0$ and $Q=\PotDann:=\PotDann_0$, (Hp~$1$) and (Hp~$2$) are satisfied. 
In particular choosing any $\varphi$ that satisfies (Hp~$3$) and (Hp~$4$) with $\varphi'(0^+)=1$, Theorem~\ref{t:finale} 
implies that $\Functeps$ in \eqref{functeps} $\Gamma$-converge to $F_1$ where $g$ is given by 
\begin{equation}
      g(s)=\inf_{(u,v)\in \calulu_s} \int_0^1 \Big(\hatf_0(v)|u'|^2+\PotDann_0(1-v)|v'|^2\Big)^{\sfrac12}\dx.
\end{equation}
Moreover $(\hatk\circ \Psi_0^{-1})^{\sfrac12}$ is convex, because it is non-decreasing and its inverse $\Psi_0(\hatk^{-1}(t^2))$ is concave. Indeed, \eqref{e:derivative hatf0-1} yields that the derivative 
of the latter function is given by
\begin{equation*}
2t\PotDann_0^{\sfrac12}(1-\hatf_0^{-1}(t^2))(\hatf_0^{-1})'(t^2)=\phi(1-t^2)
\end{equation*}
for every $t\in (0,1)$. Hence, the corresponding function $\Phi$ is strictly decreasing, and $\Phi(1^-)=0$ by \eqref{e:limite Phi =0} in Proposition~\ref{p:BCI}. Moreover, 
the condition in \eqref{e:F} is satisfied as $F(t)=\frac{\phi(t)}{2(1-t)}\in C^0([0,1))$, so that $\Phi(1^-)=0$. Hence, Theorem~\ref{t:teo la g esplicita} yields that for every $s\in (0,\sfratt)$
\begin{equation*}
    g(s)=2\int_{\hatf_0^{-1}((g'(s))^2)}^1 \Big(\frac{\hatf_0(t)\PotDann_0(1-t)}{\hatf_0(t)-(g'(s))^2}\Big)^{\sfrac12}\dt\,.
\end{equation*}
Setting $\lambda=1-(g'(s))^2$, and changing variable with $\tau=1-\hatf_0(t)$ we find
\begin{align*}
    g\big((g')^{-1}((1-\lambda)^{\sfrac12})\big)
    =\int_0^\lambda \frac{\phi(t)}{(\lambda-t)^{\sfrac12}} \dt=R(\lambda)=
    g_0\big((g_0')^{-1}((1-\lambda)^{\sfrac12})\big)\,,
\end{align*}
for every $\lambda\in [0,1)$. Arguing as in Theorem~\ref{t:hatf fixed} 
the conclusion $g=g_0$ follows at once.

If $\FailureS\neq 1$, apply the previous argument to $\widetilde{g}_0:=\FailureS^{-1}g_0$ and $\widetilde{\PotDann}_0:=\FailureS^{-2}\PotDann_0$ to find $\widetilde{\hatf_0}$ such that $\widetilde{g}_0$ is the surface energy density of the $\Gamma$-limit of the corresponding functionals $\Functeps$ with $\widetilde{Q}=\widetilde{\PotDann}:=\widetilde{\PotDann}_0$ and
$\widetilde{\hatf}:=\widetilde{\hatf}_0$. The conclusion for $g_0$
then follows by taking $\hatf:=\widetilde{\hatf}_0$, $\PotDann:=\PotDann_0$, and $Q:=\FailureS^{-2}\PotDann_0$.
\end{proof} 

\subsection{The superlinear case}
In this section we show how to reconstruct suitable cohesive laws with superlinear behaviour for small jump amplitudes. Let
\begin{itemize}
\item[(Hp~$6'$)] $g_0\in C^1((0,\infty))\cap C^0([0,\infty))$, $g_0^{-1}(0)=\{0\}$, 
$g_0$ is bounded and non-decreasing; 
\item[(Hp~$7'$)] %$g_0$ is concave, ${\displaystyle{\lim_{s\to 0^+}g'(s)}}=\infty$, 
$g'(0^+)=\infty$, $g'_0$ is %strictly decreasing and 
convex on $(0,(\sfratt)_0]$ where 
\[
(\sfratt)_0:=\sup \{s \in [0,\infty):\, g_0(s)<\lim_{s\to\infty}g_0(s)\}\,<\infty\,,
\]
and $g_0\in C^2((0,(\sfratt)_0))$ with $g_0''((\sfratt)_0^-)<0$.
\end{itemize}
In particular, $g_0$ is concave on $[0,\infty)$, and $g'_0$ is strictly decreasing on $(0,(\sfratt)_0]$.
Next, define the auxiliary function $R:[0,\infty)\to(0,g_0(\infty)]$ by
\begin{equation}\label{e:La R grande infinito}
    R(t):= g_0((g_0')^{-1}\big(t^{\sfrac12}))
  \end{equation}
Observe that $R(0^+)=g_0(\infty)$, $R(\infty)=0$, $R$ is strictly decreasing and convex, 
$R\in C^1([0,\infty))$ with  
\begin{equation*}
    R'(t)=
    \begin{cases}
        (2g_0''((g_0')^{-1}(t^{\sfrac12})))^{-1} \; & \textup{if $t\in (0,\infty)$} \\
         (2g_0''((\sfratt)_0^-))^{-1}\; & \textup{if $t=0$}.
    \end{cases}
\end{equation*}
For every $\tau\in [0,\infty)$ then define
\begin{equation}\label{e:la phi della R infinito}
    \phi(\tau):=-\frac1{\pi}\int^\infty_\tau\frac{R'(t)}{(t-\tau)^{\sfrac12}} \dt\,.
\end{equation} 
In the following proposition we establish some basic properties of the function 
$\phi$ as a solution of an Abel equation with datum $R$. Despite being a natural outcome of the theory developed in \cite{Vessella}, we have not been able to find an explicit reference there. Hence, we give a proof below.
\begin{proposition}\label{e:phi superlinear}
Assume (Hp~$6'$) and (Hp~$7'$), and let $\phi$ be defined by \eqref{e:la phi della R infinito}. Then, $\phi\in C^0([0,\infty))$ is strictly positive, non-increasing and 
satisfies the Abel equation
\begin{equation}\label{e:Abel infinito}
    R(t)=\int_{t}^{\infty}\frac{\phi(\tau)}{(\tau-t)^{\sfrac12}}\dd \tau
\end{equation}
for every $t\in [0,\infty)$.
\end{proposition}
\begin{proof}
First we observe that $\phi$ is well defined and strictly positive being %$\phi(\tau)\in [0,\infty]$ because 
$R'(t)<0$ for every $t\in [0,\infty)$. Moreover, by the change of variable 
$\tau=(g_0')^{-1}(t^{\sfrac12})$
\begin{align*}
\phi(0)=-\int_0^{\infty}\frac{R'(t)}{t^{\sfrac12}}\dt = -\int_0^{\infty}\frac{1}{2g_0''((g_0')^{-1}(t^{\sfrac12}))t^{\sfrac12}}\dt
=(\sfratt)_0<\infty\,.
\end{align*}
For every $\tau_1,\tau_2\in [0,\infty)$ with $\tau_1 <\tau_2$ we have that 
\begin{align*}
    \phi(\tau_1)&=-\frac1{\pi}\int^\infty_{\tau_1}\frac{R'(t)}{(t-\tau_1)^{\sfrac12}} \dt
    =-\frac1{\pi}\int^\infty_{\tau_2}\frac{R'(t+\tau_1-\tau_2)}{(t-\tau_2)^{\sfrac12}} \dt\\&
    \geq -\frac1{\pi}\int^\infty_{\tau_2}\frac{R'(t)}{(t-\tau_2)^{\sfrac12}} \dt=\phi(\tau_2),
\end{align*}
being $R'$ non-decreasing. Therefore, $\phi(\tau)\in (0,\infty)$ for every $\tau\in [0,\infty)$. 

To prove the continuity of $\phi$ fix $t_0\in [0,\infty)$ and $(t_j)_j$ such that $t_j\to t_0$ as $j\to \infty$. For every $\eps\in (0,1)$ we have that 
\begin{equation*}
    \phi(t_j)=\int_{t_j}^{t_j+\eps}\frac{R'(t)}{(t-t_j)^{\sfrac12}} \dt
    + \int_{t_j+\eps}^{\infty}\frac{R'(t)}{(t-t_j)^{\sfrac12}} \dt.
\end{equation*}
By Lebesgue Dominated Convergence Theorem we have that 
\begin{equation*}
    \lim_{j\to \infty} \int_{t_j+\eps}^{\infty}\frac{R'(t)}{(t-t_j)^{\sfrac12}} \dt= \int_{t_0+\eps}^{\infty}\frac{R'(t)}{(t-t_0)^{\sfrac12}} \dt,
\end{equation*}
while 
\begin{equation*}
\left|\int_{t_j}^{t_j+\eps}\frac{R'(t)}{(t-t_j)^{\sfrac12}}\dt\right| \leq 2|R'(t_j)|\eps^{\sfrac12}.
\end{equation*}
Thus
\begin{equation*}
    \limsup_{j\to \infty}|\phi(t_j)-\phi(t_0)|\leq 4|R'(t_0)|\eps^{\sfrac12}.
\end{equation*}
Taking $\eps\to 0$ we infer the continuity of $\phi$.

To obtain \eqref{e:Abel infinito}, we use Fubini's Theorem and the elementary identity
\begin{equation*}
   \int_{t_1}^{t_2} \frac{1}{(\tau-t_1)^{\sfrac12}(t_2-\tau)^{\sfrac12}}\dd \tau =\pi
\end{equation*}
for every $t_1,\,t_2\in \R$ with $t_1<t_2$. Indeed
\begin{align*}
    \int_{t}^{\infty}\frac{\phi(\tau)}{(\tau-t)^{\sfrac12}}\dd \tau&=-\frac{1}{\pi}\int_t^{\infty}\frac{1}{(\tau-t)^{\sfrac12}}\int_{\tau}^{\infty}\frac{R'(r)}{(r-\tau)^{\sfrac12}}\dd r= \\ & -\frac{1}{\pi}\int_{t}^{\infty}R'(r)\int_t^{r}\frac{1}{(\tau-t)^{\sfrac12}(r-\tau)^{\sfrac12}} \dd \tau \dd r=R(t)
\end{align*}
because $R(\tau)\to 0$ as $\tau\to \infty$.
\end{proof}

Next, we fix $\PotDann$ and determine both $\hatf$ and $Q$, and thus $\hatk$.
\begin{theorem}\label{t:potdann fixed infinito}
Let $g_0$ satisfy (Hp~$6$')-(Hp~$7$'), and let $\PotDann_0\in C^0([0,1],[0,\infty))$,
$\PotDann_0^{-1}(\{0\})=\{0\}$, $2\Psi_0(1)=g_0(\infty)$, 
%having set  $\Psi_0(x):=\int_0^x \PotDann_0^{\sfrac12}(1-t)\dt$. 
$\Psi_0$ defined as in \eqref{e:Psi} integrating $\PotDann_0$.

Then, there exist $\hatf_0$ and $Q_0\in C^0([0,1])$ such that $\hatf:=\hatf_0$, $Q:=Q_0$, 
$\PotDann:=\PotDann_0$ satisfy (Hp~$1$), (Hp~$2$) with $\FailureS=\infty$, 
$\hatk:=\FailureS^2\hatf$ satisfies (Hp~$5'$),
and for any $\varphi$ satisfying (Hp~$3$)-(Hp~$4$) with $\varphi'(0^+)=1$ the functionals $\Functeps$ in \eqref{functeps} 
$\Gamma$-converge to $F_{\infty}$ in \eqref{F0} with $g=g_0$.
\end{theorem}
\begin{proof}
Let $\phi$ be defined in \eqref{e:la phi della R infinito}, then let $\hatk^{-1}_0:[0,\infty)\to[0,1]$ be
\begin{equation*}
    \hatk^{-1}_0(t):=\Psi_0^{-1}\left(\frac{g_0(\infty)}{2}-\frac{1}{2}\int_t^{\infty} \frac{\phi(\tau)}{\tau^{\sfrac12}}\dd \tau\right)
\end{equation*}
Then $\hatk_0 \in C^0([0,1))\cap C^1((0,1))$ by Proposition~\ref{e:phi superlinear}. Moreover,
$\hatk_0(0)=0$ as $R(0)=g_0(\infty)$, $\hatk_0(t)\to \infty$ as $t\to 1$ as $2\Psi_0(1)=g_0(\infty)$, $\hatk_0'(t)>0$ for every $t\in (0,1)$ and 
\begin{equation*}
  2(\hatk_0^{-1})'(t)\big(t\PotDann_0(1-\hatk_0^{-1}(t))\big)^{\sfrac12}=\phi(t)
\end{equation*}
for every $t\in (0,1)$.
We observe that setting $\hatf(t):=\hatk_0(t)\wedge 1$, $Q(0):=0$,
$Q(t):=\frac{\PotDann_0(t)\hatf(1-t)}{\hatk_0(1-t)}$ if $t\in (0,1]$
%\begin{equation}
%Q(t):=
%\begin{cases}
%0\; & \textup{if $t=0$} \\
%    \varphi'(0^+)\frac{\PotDann_0(t)\hatf(1-t)}{\hatk_0(1-t)} \; &  \textup{if $t\in (0,1]$}
%\end{cases}
%\end{equation}
and $\PotDann:=\PotDann_0$, (Hp~$1$) and (Hp~$2$) are satisfied. In particular, 
for every $\varphi$ as in the statement, Theorem~\ref{t:finale} implies that 
$\Functeps$ in \eqref{functeps} $\Gamma$-converge to $F_{\infty}$ where $g$ is given by 
\begin{equation}
      g(s)=\inf_{(u,v)\in \calulu_s} \int_0^1 \Big(\hatk_0(v)|u'|^2+\PotDann_0(1-v)|v'|^2\Big)^{\sfrac12}\dx.
\end{equation}
Moreover $(\hatk_0\circ \Psi^{-1})^{\sfrac12}$ is convex, because it is non-decreasing 
and its inverse $\Psi(\hatk_0^{-1}(t^2))$ is concave. Indeed, the derivative of the 
latter function is 
\begin{equation*}
%   (\PotDann_0(1-\hatk_0^{-1}(t^2)))^{\sfrac12}(\hatk_0^{-1})'(t^2)\,2t=
   (\PotDann_0(1-\hatk_0^{-1}(t^2)))^{\sfrac12}(\hatk_0^{-1})'(t^2)\,2t=\phi(t^2).
\end{equation*}
Therefore, by Theorem~\ref{t:teo la g esplicita} we infer that
\begin{equation*}
    g(s)=2\int_{\hatk_0^{-1}((g'(s))^2)}^1 \Big(\frac{\hatk_0(t)\PotDann_0(1-t)}{\hatk_0(t)-(g'(s))^2}\Big)^{\sfrac12}\dt\,
\end{equation*}
and changing variable with $\hatk_0(t)=\tau$ and setting 
$\lambda=(g'(s))^2$, we find
\begin{align*}
& g\big((g')^{-1}(\lambda^{\sfrac12})\big)=2\int^{\infty}_\lambda
\Big(\frac{\tau\,\PotDann_0(1-\hatk_0^{-1}(\tau))}{\tau-\lambda}\Big)^{\sfrac12}(\hatk_0^{-1})'(\tau)\dd\tau \\
& =\int_\lambda^\infty \frac{\phi(\tau)}{(\tau-\lambda)^{\sfrac12}} \dd\tau=R(\lambda)=
    g_0\big((g_0')^{-1}(\lambda^{\sfrac12})\big)\,,
\end{align*}
for every $\lambda\in [0,\infty)$. Arguing as in Theorem~\ref{t:hatf fixed} the conclusion $g=g_0$ follows at once.
\end{proof}
Finally, with fixed $\hatk$ we determine $\hatf$, $Q$ and $\PotDann$.
\begin{theorem}\label{t:hatf fixed infinito}
Assume that $g_0$ satisfy (Hp~$6'$)-(Hp~$7'$), and $\hatk_0$ is such that $\hatk_0(0)=0$,
$\hatk_0(1^-)=\infty$, $\hatk_0^{\sfrac12}\in C^1([0,1))$ with $\hatk_0'(t)>0$ for every 
$t\in(0,1)$, concave in a left neighborhood of $t=1$, and such that
$(\hatk_0^{\sfrac12}(t))'\phi(\hatk_0(t))$ is decreasing and infinitesimal as $t\to1$, 
where $\phi$ is defined in \eqref{e:la phi della R infinito} in particular
$\hatk_0$ satisfies (Hp~$5'$).

Then there exist $\PotDann_0$ and $Q_0\in C^0([0,1])$ such that $\hatf(t):=\hatk_0(t)\wedge 1$, 
$Q:=Q_0$, $\PotDann:=\PotDann_0$ satisfy (Hp~$1$)-(Hp~$2$) with $\FailureS=\infty$; and for any 
$\varphi$ satisfying (Hp~$3$)-(Hp~$4$) with $\varphi'(0^+)=1$ the functionals $\Functeps$ in 
\eqref{functeps} $\Gamma$-converge to $F_{\infty}$ in \eqref{F0} with $g=g_0$.
\end{theorem}
\begin{proof}
Let $\phi$ be defined in \eqref{e:la phi della R infinito}, and then set for every $t\in [0,1)$
\begin{equation}\label{e:PotDann zero infinito}
    \PotDann_0(1-t):=\big((\hatk_0^{\sfrac12}(t))'\phi(\hatk_0(t))\big)^2\,,
\end{equation}
and $\PotDann_0(0):=0$. Note that $\PotDann_0\in C^0([0,1])$ by the assumptions on $\hatk_0$.
With the choices $\hatf(t):=\hatk_0(t)\wedge 1$, $Q(0):=0$ and $Q(1-t):=\frac{\PotDann_0(1-t)\hatf(t)}{\hatk_0(t)}
$ for $t\in[0,1)$, (Hp~$1$) and (Hp~$2$) are satisfied as $\phi$ is strictly increasing and $\hatk_0^{\sfrac12}$ is concave in a left neighborhood of $t=1$.
Moreover, \eqref{e:FailureS def} holds with $\FailureS=\infty$. 
In particular, choosing any $\varphi$ that satisfies (Hp~$3$) and (Hp~$4$) with $\varphi'(0^+)=1$, 
by Theorem~\ref{t:finale} $(\Functeps)_\eps$ $\Gamma$-converges to $F_\infty$, where the surface energy 
density $g$ is given by 
\begin{equation*}
      g(s):=\inf_{(u,v)\in \calulu_s} \int_0^1 \Big(\hatk_0(v)|u'|^2+\PotDann_0(1-v)|v'|^2\Big)^{\sfrac12}\dx.
\end{equation*}
The rest of the proof follows exactly as in Theorem~\ref{t:potdann fixed infinito}.
\end{proof}

\subsection{Examples}\label{ss:esempi g esplicite}
In this section, we apply Theorems~\ref{t:teo la g esplicita Dugdale}, \ref{t:hatf fixed}, \ref{t:potdann fixed} and \ref{t:potdann fixed infinito} to find closed form solutions for phase-field models whose $\Gamma$-limit have as surface energy density assigned cohesive laws $g$ frequently used in applications,~\Fig{fig_Ex}. From a numerical point of view the problem has already been discussed extensively in the papers \cite{Wu2017,Wu2018b,Feng2021,Lammen2025,Wu2025}.
%We remark explicitly that the reconstruction problem for the Dugdale cohesive law has been first addressed in \cite[Section~7.1]{Conti2016} and in \cite{LMC24} with the choice $\varphi(t)=t\wedge 1$.
We note that the reconstruction problem for the Dugdale cohesive law has been addressed in \cite[Section~7.1]{Conti2016} and \cite{Lammen2025} using a truncated degradation function of the form $\varphi(t)=1\wedge t$, as discussed in \cite[formula (1.2)]{Alessi2025a}.
In contrast, we obtain here Dugdale cohesive law by using the smooth degradation function $\varphi(t)=\frac t{1+t}$, as shown in \cite[formula (1.3)]{Alessi2025a}. Consequently, in a one-dimensional setup, the mechanical global response corresponding to the Dugdale cohesive law with the truncated degradation function is obtained only in the limit as $\eps\to 0$, since the critical stress also depends on the regularizing length $\eps$ \cite[Fig. 2]{Lammen2025}. By employing the smooth degradation function, as done in this work, the regularizing length $\eps$ only influences the phase-field profile, ensuring that the mechanical global response is immediately and correctly described. Further details about this point can be found in the application-oriented part of this work,~\cite{Alessi2025c}.

We will use the notation of section~\ref{s:sezione 5.3}.
%For the sake of simplicity, in every constructed phase-field model $\Functeps$ (see \eqref{functeps})
Recall that we have assumed (Hp~$1$)-(Hp~$4$), and moreover (Hp~$5'$). To be consistent with section~\ref{s:sezione 5.3},
we suppose that the function $\varphi$ satisfies $\varphi'(0^+)=1$. For the sake of simplicity,
if $\hatk(1^-)<\infty$ we shall choose $Q=\PotDann$ so that $g'(0^+)=1$. This is clearly not a limitation
using the rescaling argument employed in the proofs of Theorems~\ref{t:hatf fixed} and \ref{t:potdann fixed}.
%and $\PotDann=Q$. In particular,
%the cohesive laws that will be considered satisfy the condition $g'(0)=1$ (cf. \eqref{e:strenght}).
%Clearly, this is not a restriction and can be easily taken care of.

\newcommand{\myfig}[1]{\includegraphics[page=#1,scale=0.8,trim=3mm 0 0 1mm, clip]{ACF_2-2-Examples}}
\newcommand{\mysubref}[1]{(\subref{#1})}

\begin{figure}[h!]
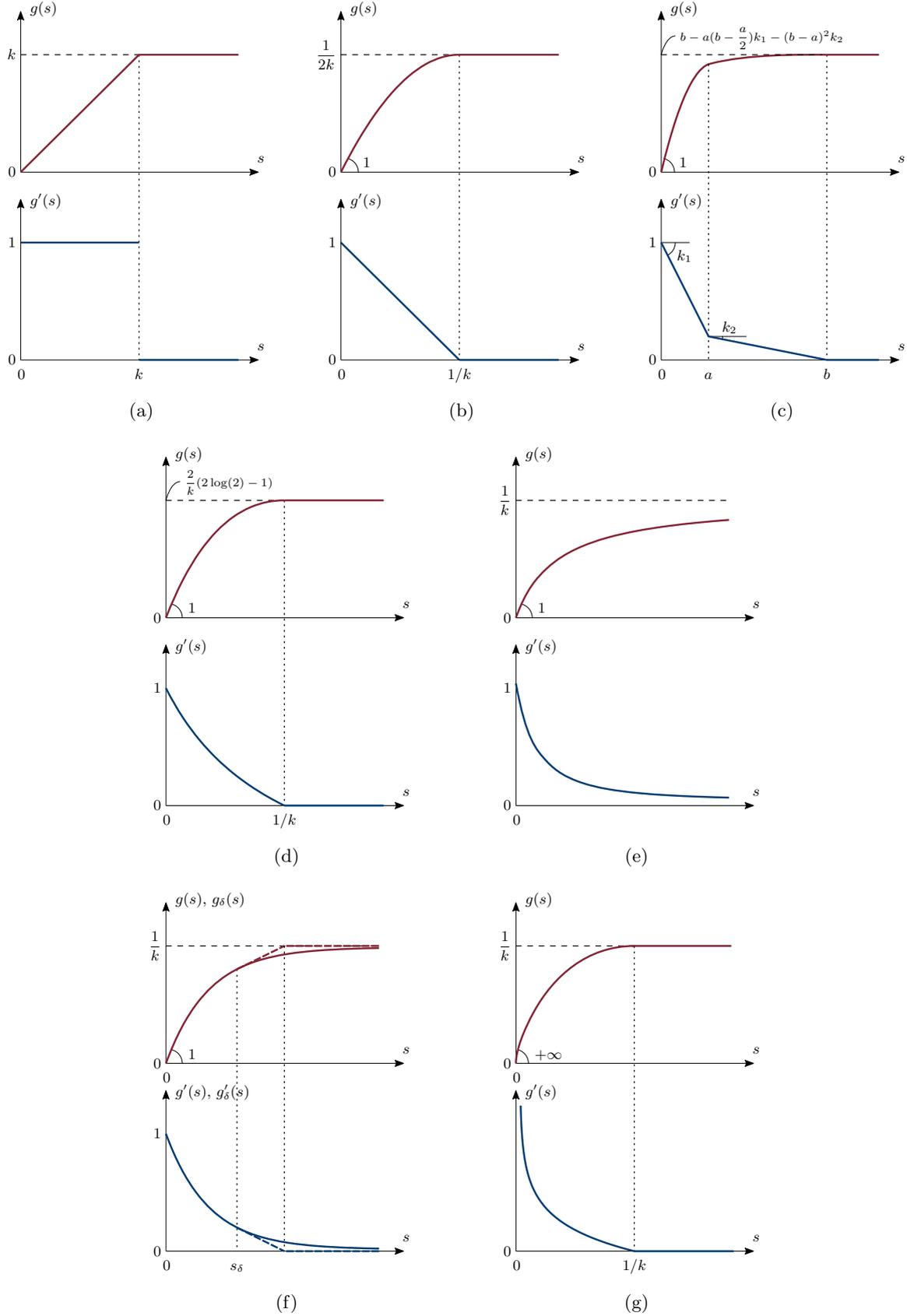

  \centering
  \begin{subfigure}{0.32\linewidth}
    \centering
      \myfig{1}
    \caption{} \label{fig_E_D} 
  \end{subfigure}
  \hfill
  \begin{subfigure}{0.32\linewidth}
    \centering
      \myfig{2}
    \caption{} \label{fig_E_L} 
  \end{subfigure}
  \hfill
  \begin{subfigure}{0.32\linewidth}
    \centering
      \myfig{3}
    \caption{} \label{fig_E_L2} 
  \end{subfigure}
  \\[1ex]
  \begin{subfigure}{0.32\linewidth}
    \centering
      \myfig{4}
    \caption{} \label{fig_E_H} 
  \end{subfigure}
  \qquad
  \begin{subfigure}{0.32\linewidth}
    \centering
      \myfig{5}
    \caption{} \label{fig_E_H2} 
  \end{subfigure}  
  \\[1ex]
  \begin{subfigure}{0.32\linewidth}
    \centering
      \myfig{6}
    \caption{} \label{fig_E_E} 
  \end{subfigure}
  \qquad
  \begin{subfigure}{0.32\linewidth}
    \centering
      \myfig{7}
    \caption{} \label{fig_E_Log} 
  \end{subfigure}  
\caption{Cohesive laws considered in the examples of section~\ref{ss:esempi g esplicite}: \mysubref{fig_E_D} Dugdale, \mysubref{fig_E_L} Linear, \mysubref{fig_E_L2} Bilinear, \mysubref{fig_E_H} hyperbolic, \mysubref{fig_E_H2} quadratic hyperbolic, \mysubref{fig_E_E} exponential and \mysubref{fig_E_Log}~logarithmic cohesive softening laws.}
\label{fig_Ex} 
\end{figure}

\FloatBarrier

\subsubsection{Dugdale's model}
Consider the Dugdale's model given by the cohesive law 
\begin{equation}
    g(s)=
    \begin{cases}
        s \; \textup{ if $s\in[0,k]$} \\
        k \; \textup{ if $s\in(k,\infty)$}
    \end{cases}
\end{equation}
where $k\in (0,\infty)$. We observe that if $\hatf$ and $\PotDann$ are such that $\hatf\in C^1((0,1))$ with $\hatf'(t)>0$ for every $t\in (0,1)$ and
\begin{equation*}
   \Psi(\hatf^{-1}(\tau^2))=\frac{k}{\pi}(\sin^{-1}(\tau)-\tau(1-\tau^2)^{\sfrac12})\,
\end{equation*}
for every $\tau\in [0,1]$, then an easy computation gives $\Phi(x)=2k\hatf^{\sfrac12}(x)$ for every 
$x\in (0,1)$, where $\Psi$ and $\Phi$ are respectively defined in \eqref{e:Psi} and \eqref{e:la psi piccola}. Therefore, if $\Functeps$ are the functionals defined in \eqref{functeps} with $Q=\omega$, then
Theorem~\ref{t:teo la g esplicita Dugdale} implies that ${\displaystyle{\Gamma\textup{-} \lim_{\eps \to 0}\Functeps}}$ is a functional as 
in \eqref{F0} with $g$ as surface energy density. In particular we obtain the following
\begin{itemize}
    \item[$\cdot$] if $\omega(1-t)=k^2(1-t)^2$ then $\hatf^{-1}(t)=1-\left[1-\frac{2}{\pi}\left(\sin^{-1}(t^{\frac{1}{2}})-(t-t^2)^{\frac{1}{2}}\right)\right]^{\frac{1}{2}}$
    \item[$\cdot$] if $\omega(1-t)=\frac{9k^2}{16}(1-t)$ then $\hatf^{-1}(t)=1-\left[1-\frac{2}{\pi}\left(\sin^{-1}(t^{\frac{1}{2}})-(t-t^2)^{\frac{1}{2}}\right)\right]^{\frac{2}{3}}$
\end{itemize}

\subsubsection{Linear softening}
Let $k\in (0,\infty)$, and consider softening laws of the form
 \begin{equation*}
        g'(s)=
        \begin{cases}
            1-ks \; & \textup{ if $s\in[0,\frac{1}{k}]$} \\
            0 \; & \textup{ if $s\in(\frac{1}{k},\infty)$}
        \end{cases}
    \end{equation*}
One can easily verify that the corresponding $g$ satisfies (Hp~$6$)-(Hp~$7$), $R(t)=\frac{t}{2k}$ fulfills (Hp~$8$), and that $\phi(t)=\frac{t^{\sfrac12}}{k\pi}$ by \eqref{e:la phi della R}. In particular, we deduce that
\begin{itemize}
        \item[$\cdot$]  if $\hatf(t)=t^2$ then $\omega(1-t)=\frac{1-t^2}{k^2\pi^2}$,
        \item[$\cdot$] if $\omega(t)=\frac{t^2}{4k^2}$ then $\hatf^{-1}(t)=1-\big(1-\frac{2}{\pi}((t-t^2)^{\sfrac12}+\cos^{-1}((1-t)^{\sfrac12}))\big)^{\sfrac12}$,
%    \end{itemize}
%For engineering reasons \FC{chiedere a Roberto} it is important to consider also linear damage potential %$\PotDann$:
%\begin{itemize}
        \item[$\cdot$] if $\omega(t)=\frac{9t}{64k^2}$ then $\hatf^{-1}(t)=1-\big(1-\frac{2}{\pi}((t-t^2)^{\sfrac12}+\cos^{-1}((1-t)^{\sfrac12})\big)^{\sfrac{2}{3}}$.
    \end{itemize}
    We remark that in case $\hatf(t)=t^2$ we obtain exactly the same damage potential $\PotDann$ as in \cite[Section 4.1]{Feng2021}
\subsubsection{Bilinear softening}
In this case we fix $a,b,k_1,k_2 \in (0,\infty)$ such that $a<b$, $k_2<k_1$ and 
 \begin{equation*}
        g'(s)=
        \begin{cases}
            1-k_1s \; & \textup{ if $s\in[0,a]$} \\
            (1-k_1a+k_2a)-k_2s \; & \textup{ if $s\in(a,b]$} \\
            0 \; & \textup{ if $s\in(b,\infty)$}
        \end{cases}
  \end{equation*}
    is continuous. Therefore, the corresponding $g$ satisfies (Hp~$6$)-(Hp~$7$) and simple calculations 
    yield $R\in W^{1,\infty}((0,1))$ with
    \begin{equation*}
       R'(t)=
       \begin{cases}
           \frac{1}{2k_1} \; & \textup{if $t\in[0,1-(1-k_1 a)^2]$} \\
           \frac{1}{2k_2} \; & \textup{ if $t\in(1-(1-k_1 a)^2,1]$}.
       \end{cases}
    \end{equation*}
In particular, $R$ fulfills (Hp~$8$), and if $\hatf(t)=t^2$ then
\begin{equation*}
            \omega(1-t)=
            \begin{cases}
              \left(\frac{(1-t^2)^{\sfrac12}}{k_1\pi}+(\frac{1}{k_2\pi}-\frac{1}{k_1\pi})((1-k_1a)^2-t^2)^{\sfrac12}\right)^2 & \textup{ if $t\in[0,1-k_1a]$} \\
                  \frac{1-t^2}{k_1^2\pi^2} & \textup{ if $t\in(1-k_1a,1]$}.\\
            \end{cases}
        \end{equation*}

\subsubsection{Hyperbolic softening}
Let $k\in (0,\infty)$ and consider hyperbolic softening laws of the form
\begin{equation*}
        g'(s)=
        \begin{cases}
            \frac{2}{1+ks}-1 \; & \textup{if $s\in[0,\frac{1}{k}]$} \\
            0 \; & \textup{if $s\in(\frac{1}{k},\infty)$}.
        \end{cases}
    \end{equation*}
Then 
\begin{equation*}
        g(s)=
        \begin{cases}
            \frac{2}{k}\log(1+ks)-s \; & \textup{if $s\in[0,\frac{1}{k}]$} \\
            \frac{1}{k}(2\log(2)-1) \; & \textup{if $s\in(\frac{1}{k},\infty)$}
        \end{cases}
    \end{equation*}
satisfies (Hp~$6$)-(Hp~$7$) and the corresponding function $R$ fulfills (Hp~$8$) with 
\begin{equation*}
   R'(t)=\frac{1}{k(1+(1-t)^{\frac{1}{2}})^2} 
\end{equation*}
for every $t\in [0,1]$. In particular, by an easy computation and Theorem~\ref{t:hatf fixed}, we obtain, with $\hatf(t)=t^2$, that
\begin{equation*}
    \omega(1-t)=\begin{cases}
    \frac{(2t^2\log(t)+1-t^2)^2}{k^2\pi^2(1-t^2)^3} &  t\in [0,1) \\
    0 & t=1
    \end{cases}.
\end{equation*}

\subsubsection{Quadratic hyperbolic softening}
Let $k\in (0,\infty)$ and consider quadratic hyperbolic softening laws of the form $g'(s)= \frac{1}{(1+ks)^2}$.
Then $g$ satisfies (Hp~$6$)-(Hp~$7$) with
\begin{equation}
    R'(t)=\frac{4k}{(1-t)^{\frac{1}{4}}}
\end{equation}
for every $t\in [0,1)$. In particular $R'\in L^p((0,1))$ for every $p\in [1,4)$ and thus $R$ fulfills (Hp~$8$). From a simple computation and Theorem \ref{t:hatf fixed} we can infer that if $\hatf(t)=t^2$ then $\omega(1-t)=\phi(1-t^2)^2$ where
\begin{equation}
   \phi(t)= \frac{16k}{3\pi(1-t)^{\frac{1}{2}}}\left[(1-t)^{\frac{3}{4}} {}_{2}F_1(1/2,3/4;7/4;1)-{}_2F_1(1/2,3/4;7/4;1/(1-t))\right]
\end{equation}
and ${}_2F_1(\cdot \; ,\cdot \;;\cdot \;;\cdot)$ is the Hypergeometric function.

\subsubsection{Exponential softening}\label{s:exponential example}
Let $k\in (0,\infty)$ and consider exponential softening laws of the form $g'(s)=e^{-ks}$, namely $g(s)=\frac{1}{k}(1-e^{-ks})$
for $s\geq 0$. It is easy to verify that $g$ fulfills (Hp~$6$)-(Hp~$7$) with $R(t)=\frac{1}{k}(1-(1-t)^{\sfrac12})$. Note that $R'\notin L^p((0,1))$ for any $p>2$, and thus $R$ does not satisfy (Hp~$8$) so that we can apply neither Theorem~\ref{t:hatf fixed} nor \ref{t:potdann fixed}. Assumption (Hp~$8$) ensures the continuity of the function $\phi$ on $[0,1]$, which indeed fails in this particular case as 
$\phi(t)= \frac{1}{k\pi}\cosh^{-1}\left(\frac{1}{(1-t)^{\sfrac12}}\right)$. 
Therefore, using the constructions in Theorem~\ref{t:hatf fixed} and \ref{t:potdann fixed} we would get that either $\hatf$ or $\PotDann$ is not in $C^0([0,1])$. On the other hand, the continuity of both $\hatf$ and $\PotDann$ is a crucial assumption 
%in the  ${\displaystyle{\Gamma}}$-convergence argument 
in Theorem~\ref{t:finale} (cf. (Hp~$1$)-(Hp~$2$)). 

To overcome this problem, we slightly modify $g$ as follows:
let $\delta>0$ and $s_{\delta}\to \infty$ as $\delta\to 0^+$, we 
linearize $g'$ in $(s_{\delta},\infty)$ to obtain a new cohesive law $g_{\delta}$ given by
\begin{equation*}
g'_{\delta}(s):=
\begin{cases}
   e^{-ks} \; &\textup{if $s\in[0,s_{\delta}]$} \\
    (e^{-ks_{\delta}}-ke^{-ks_{\delta}}(s-s_{\delta}))\lor 0 \; &\textup{if $s\in(s_{\delta},\infty)$}.
\end{cases}
\end{equation*}
For every $\delta\in (0,1)$, $g_{\delta}$ satisfies (Hp~$6$)-(Hp~$7$) with 
 \begin{equation*}
       R_{\delta}(t)=
       \begin{cases}
          \frac{1}{k}(1-(1-t)^{\sfrac12}) \; & \textup{if $t\in[0,\hat s_{\delta}]$} \\
          \frac{1}{k}(1-(1-\hat s_{\delta})^{\sfrac12}) \; & \textup{if $ t\in(\hat s_{\delta},1]$},
       \end{cases}
    \end{equation*}
for some $\hat s_{\delta}$ such that $\hat s_{\delta}\to 1$ as $\delta \to 0^+$. Moreover, $R_{\delta}$ fulfills (Hp~$8$) and $g_{\delta}\to g$ uniformly and monotonically on $[0,\infty)$. In particular, one may choose $\hat s_{\delta}=1-\delta$ to obtain
\begin{equation*}
  \phi_{\delta}(t)=
  \begin{cases}
      \frac{1}{k\pi}\cosh^{-1}\left(\frac{1}{(1-t)^{\sfrac12}}\right)  & \hskip-1cm\textup{if $t\in[0,1-\delta]$} \\
      \frac{1}{k\pi}\left(\cosh^{-1}\left(\frac{1}{(1-t)^{\sfrac12}}\right)-\cosh^{-1}\left((\frac{\delta}{1-t})^{\sfrac12}\right)+\Big(\frac{\delta-(1-t)}{\delta}\Big)^{\sfrac12}\right) & \\
      & \hskip-1cm\textup{if $t\in(1-\delta,1]$}\,,
  \end{cases}
\end{equation*}
and hence get from Theorem~\ref{t:hatf fixed} for $\hatf(t)=t^2$ 
\begin{equation*}
            \omega_{\delta}(1-t)=
            \begin{cases}
                \frac{1}{k^2\pi^2}\left(\cosh^{-1}(\frac{1}{t})-\cosh^{-1}(\frac{\delta^{\sfrac12}}{t})+\left(\frac{\delta-t^2}{\delta}\right)^{\sfrac12}\right)^2 & \textup{if $t\in[t, \delta^{\sfrac12}]$} \\
                 \frac{1}{k^2\pi^2}(\cosh^{-1})^2(\frac{1}{t}) & \textup{if $t\in(\delta^{\sfrac12},1]$}. \\
            \end{cases}
        \end{equation*}
Let then $\mathcal{F}_{\eps,\delta}$ be defined as in \eqref{functeps} with 
$\hatf(t)=t^2$ and $Q=\omega=\omega_{\delta}$.
Theorem~\ref{t:finale} implies that ${\displaystyle{\Gamma\textup{-} \lim_{\eps \to 0}\mathcal{F}_{\eps,\delta}}}$ is a functional as 
in \eqref{F0} with $g_\delta$ as surface energy density, and diffuse energy densities independent from $\delta$.
Noting that $g_{\delta_2}\geq g_{\delta_1}$ for $0<\delta_1<\delta_2<1$,  \cite[Proposition~5.4]{Dalmaso1993} yields that
${\displaystyle{\Gamma \textup{-}\lim_{\delta \to 0}\left( \Gamma\textup{-} \lim_{\eps \to 0}\mathcal{F}_{\eps,\delta}\right)}}$ is a functional as in \eqref{F0} with surface energy density $g(s)=\frac{1}{k}(1-e^{-ks})$, $s\geq 0$.

In conclusion, we observe that $\omega_{\delta}$ coincides with the damage potential found in \cite[Section~4.2]{Feng2021} on $[\delta^{\sfrac{1}{2}},1]$ for every $\delta \in (0,1)$, using the identity $\cosh^{-1}(1/t)=\tanh^{-1}((1-t^2)^{\sfrac12})$.

\subsubsection{Logarithmic softening}
In this section we provide an application of Theorem~\ref{t:potdann fixed infinito} in order to approximate cohesive
laws with logarithmic behaviour for small jump amplitudes. Let $k\in (0,\infty)$ and consider
%softening laws of the form
\begin{equation*}
        g'(s)=
        \begin{cases}
            -\log(ks) \; & \textup{if $s\in(0,1/k]$} \\
            0 \; & \textup{if $s\in(1/k,\infty)$},
        \end{cases}
    \end{equation*}
then
\begin{equation*}
        g(s)=
        \begin{cases}
           s(1-\log(ks))\; & \textup{if $s\in[0,\frac{1}{k}]$} \\
            \frac{1}{k} \; & \textup{if $s\in(\frac{1}{k},\infty)$}
        \end{cases}
    \end{equation*}
satisfies (Hp~$6'$)-(Hp~$7'$) and the corresponding function $R$ is given by
\begin{equation*}
   R'(t)=-\frac{1}{2k}e^{-t^{\sfrac12}}
\end{equation*}
for every $t\in [0,\infty)$. In particular, by \eqref{e:la phi della R infinito}, we have that
\begin{equation*}
    \phi(t)=\frac{1}{2k\pi}\int_t^{\infty}\frac{e^{-x^{\sfrac12}}}{(x-t)^{\sfrac12}}\dx
\end{equation*}
and from Theorem \ref{t:potdann fixed infinito}, if $\omega(t)=\frac{9}{16k}t$, then we obtain 
\begin{equation*}
    f^2(t):=\frac{\hatf(t)}{Q(1-t)}=\frac{16k\hatk(t)}{9(1-t)}
\end{equation*}
where 
\begin{equation*}
    \hatk^{-1}(t)=1-\left[k\int_t^{\infty}\frac{\phi(x)}{x^{\sfrac12}}\dx\right]^{\frac{2}{3}}.
\end{equation*}

\section*{Acknowledgments}
The second and third authors thanks Sergio Conti for useful conversations on the subject of the paper.

\section*{Data availability statement}

Data sharing not applicable to this article as no datasets were generated or analysed during the current study.

\section*{Declarations}
The authors have no relevant financial or non-financial interests to disclose.

The authors have no conflicts of interest to declare that are relevant to the content of this article.

All authors certify that they have no affiliations with or involvement in any organization or entity with any financial interest or non-financial interest in the subject matter or materials discussed in this manuscript.

The authors have no financial or proprietary interests in any material discussed in this article.

\section*{Funding}
The first author has been supported by the European Union - Next Generation EU, Mission 4 Component 1 CUP G53D23001140006, codice 20229BM9EL, PRIN2022 project: “NutShell - NUmerical modelling and opTimisation of SHELL Structures Against Fracture and Fatigue with Experimental Validations”.
The first author also acknowledges the Italian National Group of Mathematical Physics INdAM-GNFM.

The second and third authors have been supported by the European Union - Next Generation EU, Mission 4 Component 1 CUP 2022J4FYNJ, PRIN2022 project ``Variational methods for stationary and evolution problems with singularities and interfaces''. The second and third authors are members of GNAMPA - INdAM.

% ____________________________________________________________________

% ==============================================================================
% \bibliographystyle{plainabb}
%\bibliographystyle{plain}
%\bibliographystyle{elsarticle-num-names}
\bibliographystyle{alpha-noname}
\bibliography{%
  ./Research-Fracture-2023-Cohesive-Focardi-Colasanto-Alessi,%
  ./cfi}

\end{document}